\documentclass[11pt]{article}

\usepackage{amsmath,amssymb,amsthm}
\usepackage{hyperref}
\usepackage{enumitem, xcolor}

\DeclareMathOperator{\Irr}{Irr}
\DeclareMathOperator{\cd}{cd}

\DeclareMathOperator{\Sub}{Sub}
\DeclareMathOperator{\Syl}{Syl}
\DeclareMathOperator{\SL}{SL}
\DeclareMathOperator{\PSL}{PSL}
\DeclareMathOperator{\SU}{SU}
\DeclareMathOperator{\PSU}{PSU}
\DeclareMathOperator{\Sz}{Sz}
\DeclareMathOperator{\St}{St}

\newcommand{\bN}{\mathbf{N}}
\newcommand{\bC}{\mathbf{C}}
\newcommand{\bZ}{\mathbf{Z}}

\newcommand{\bO}{\mathbf{O}}
\newcommand{\QQ}{\mathbb{Q}}
\newcommand{\Gal}{\mathrm{Gal}}

\numberwithin{equation}{section}

\newtheorem{theorem}{Theorem}[section]
\newtheorem{conjecture}[theorem]{Conjecture}
\newtheorem*{conjA}{Conjecture~A}
\newtheorem*{conjB}{Conjecture~B}
\newtheorem{lemma}[theorem]{Lemma}
\newtheorem{proposition}[theorem]{Proposition}
\newtheorem{corollary}[theorem]{Corollary}
\newtheorem{remark}[theorem]{Remark}
\newtheorem{question}[theorem]{Question}

\title{Alperin's Main Problem of Block Theory}

\author{Alexander Moret\'o\thanks{The author's research is supported by Ministerio de Ciencia e
Innovaci\'on (Grant PID2022-137612NB-I00 funded by
MCIN/AEI/10.13039/501100011033 and ``ERDF A way of making Europe'').}\\
Department of Mathematics\\
University of Valencia\\
\texttt{alexander.moreto@uv.es}}

\date{}

\begin{document}

\maketitle

\begin{center}
\emph{Dedicated to the memory of Jon Alperin and Carlo Casolo}
\end{center}

\begin{abstract}
This paper proposes a conjectural framework for Alperin's Main Problem of
Block Theory from 1976. The character sets considered here are
defined by nonvanishing at given elements, not only by degree conditions.
From this point of view, McKay's conjecture is usually recovered as a first
degree-level consequence. The guiding idea is that the right local objects governing
character values are not, in general, the sets $\Irr_{p'}(G)$ and the
normalizers of Sylow $p$-subgroups, but rather the sets $\Irr^x(G)$ of
irreducible characters not vanishing at a given element $x$, together with
the subnormalizer subgroup $\Sub_G(x)$. I state the basic conjectures of this theory, propose stronger versions, and
verify the main conjectures in several families, including the simple groups
with TI Sylow $p$-subgroups. I also show how this perspective reorganizes several
classical questions in character theory.
\end{abstract}

\tableofcontents

\section{Introduction}

This paper lies at the intersection of two ideas that have remained
largely unexplored: Alperin's formulation of the main problem of block
theory, and Casolo's notion of subnormalizer. Although Alperin explicitly
singled out the determination of character values in terms of $p$-local
subgroups as the main problem of block theory, later developments largely
focused on other local-global questions, and Problem~A itself (see below) came to seem
too ambitious to be meaningfully approachable. Likewise, Casolo's
subnormalizer, despite its natural group-theoretic significance, has
remained essentially absent from the representation theory of finite groups literature.
The conjectures proposed in this paper suggest that these two ideas are in
fact closely connected.

Recall that in his Santa Cruz article \cite{alp} Alperin wrote:

\begin{quote}
Brauer's work on block theory, stretching over decades, strongly suggests
the following problem as a reasonable choice for the main problem of the
subject.

\medskip

\noindent
Problem A. Give rules which determine the values of the characters of $G$
in terms of the $p$-local subgroups of $G$.
\end{quote}

He then added that McKay's conjecture is a prime example of the kind of
rule one should expect. But, important as it is, the McKay conjecture only
concerns character values at the identity element. The same is true, in
essence, of the usual classical local-global conjectures in character
theory: they deal with degrees, heights, or counting problems, and thus
only see a very small part of the problem Alperin had in mind. By
contrast, Problem~A asks for a local understanding of \emph{all} character
values, especially at $p$-singular elements. In fact, this was the second
time that he singled out this problem as the main problem of block theory
\cite{alp2}. Now, half a century later, my goal in this paper is to propose a conjectural
framework for approaching this problem.

The starting point of the present paper is the observation that the
classical set $\Irr_{p'}(G)$ is often not the right one for this purpose.
Given an element $x\in G$, it is much more natural to consider
$$
\Irr^x(G)=\{\chi\in\Irr(G)\mid \chi(x)\neq 0\}
$$
if we want to study the values of the irreducible characters at $x$.
When $x$ is a $p$-element, this set contains $\Irr_{p'}(G)$, but is
typically much richer. The guiding principle of this paper is that
$\Irr^x(G)$ should replace $\Irr_{p'}(G)$ in the ``correct'' formulation of
many local-global problems.

A first indication comes from the following phenomenon. Let $p$ be a prime
and let $P\in \Syl_p(G)$. A $p$-element $x\in G$ will be called
\emph{picky} if it lies in a unique Sylow $p$-subgroup of $G$. These
elements arise naturally in group theory and turn out to be closely related
to character values. The first conjecture is the following.

\bigskip
\begin{conjA}
\textup{(The picky conjecture).} Let $G$ be a finite group, let $p$ be a
prime, let $P\in\Syl_p(G)$, and let $x\in P$ be a picky element. Then
there exists a bijection
$$
f:\Irr^x(G)\longrightarrow \Irr^x(\bN_G(P))
$$
such that
\begin{enumerate}[label=(\arabic*)]
\item $\chi(1)_p=f(\chi)(1)_p$ for all $\chi\in \Irr^x(G)$;
\item $\mathbb{Q}(\chi(x))=\mathbb{Q}(f(\chi)(x))$ for all
      $\chi\in \Irr^x(G)$.
\end{enumerate}
\end{conjA}

The strong picky conjecture is the version in which Condition~(2) is
replaced by the condition $f(\chi)(x)=\pm\chi(x)$ for all
$\chi\in\Irr^x(G)$.

Thus, for picky elements, the normalizer of a Sylow $p$-subgroup appears
as the local subgroup controlling the relevant character values. This
already goes far beyond the classical McKay framework: it concerns actual
values of characters at nontrivial elements, not merely degrees.

However, picky elements are only the first case of a broader picture.
The second key observation is that the normalizer of a Sylow subgroup is
not always the correct subgroup attached to a $p$-element. Following
Casolo \cite{cas89}, given a subgroup $H\leq G$, I write
$$
S_G(H)=\{\,g\in G\mid H\trianglelefteq\trianglelefteq \langle H,g\rangle\,\}
$$
to denote the subnormalizer subset of $H$. This subset does not need to be a
subgroup, and I put
$$
\Sub_G(H)=\langle S_G(H)\rangle
$$
to denote the \emph{subnormalizer subgroup}.
If $H=\langle x\rangle$ is cyclic, I simply write $\Sub_G(x)$. The point
is that, for a $p$-element $x$, $\Sub_G(x)$ contains $\bN_G(P)$, and it
coincides with $\bN_G(P)$ exactly when $x$ is picky
(see Lemma~\ref{lem:subnormalizer-picky}). This leads to the more general
conjecture.

\bigskip
\begin{conjB}
\textup{(The subnormalizer conjecture).} Let $G$ be a finite group, let $p$
be a prime, and let $x$ be a $p$-element of $G$. Then there exists a
bijection
$$
f:\Irr^x(G)\longrightarrow \Irr^x(\Sub_G(x))
$$
such that
\begin{enumerate}[label=(\arabic*)]
\item $\chi(1)_p=f(\chi)(1)_p$ for all $\chi\in \Irr^x(G)$;
\item $\mathbb{Q}(\chi(x))=\mathbb{Q}(f(\chi)(x))$ for all
      $\chi\in \Irr^x(G)$.
\end{enumerate}
\end{conjB}

The strong subnormalizer conjecture is obtained analogously, replacing the
equality of fields by equality up to sign.

In this form, the picky conjecture becomes a special case of a more
general principle: the ``local" subgroup (not local in the classical sense) that encodes the values of
irreducible characters at a given $p$-element $x$ is not necessarily the
Sylow normalizer, but rather the subnormalizer subgroup $\Sub_G(x)$.

The aim of this paper is to develop this point of view in a systematic
way. I begin with the elementary forms of the picky and subnormalizer
conjectures stated above, since they already capture the main conceptual
shift. I then propose stronger formulations, including versions in terms of
$p$-sections and reduced $p$-sections, which appear to be closer to the true
level of generality of the theory. After that, I discuss how this framework is
related to Evseev's strengthening of Isaacs--Navarro, to Brou\'e-type
phenomena, and to several classical problems in character theory. In
particular, I explain why the passage from $\Irr_{p'}(G)$ to $\Irr^x(G)$
and from $\bN_G(P)$ to $\Sub_G(x)$ drastically changes the meaning and
scope of several familiar conjectures.

In this formulation, several classical local-global conjectures appear as
degree-level consequences of statements about character values at specific
elements. This is why the sets $\Irr^x(G)$ and the subnormalizers $\Sub_G(x)$
lead to stronger, and in some cases more accurate, formulations.

The paper is organized as follows. In Section~2 I introduce the global picky
conjecture, starting from the work \cite{mmm} and from Alperin's
Conjecture~C for TI Sylow subgroups. Section~3 proves the main conjectures
for the simple groups with noncyclic TI Sylow $p$-subgroups using explicit
character-theoretic arguments; this is one of the main results proved in the
present paper. In Section~4 I introduce the subnormalizer conjecture and its
stronger forms in terms of $p$-sections and reduced $p$-sections. Section~5
relates the picky point of view to Evseev's refinement of the McKay conjecture
and to related Brou\'e-type questions. Section~6 compares the picky conjecture
with the Eaton--Moret\'o conjecture. Section~7 considers elements of mixed
order and a Hall $\pi$-version suggested by counterexamples to the
$\pi$-McKay statement. Section~8 discusses further consequences and questions,
including the relation between picky elements, subnormalizers and other
local-global problems in character theory. For more details on the genesis of this article
and its relation with earlier work, I refer the reader to \cite{mor}. 

\section{The global picky conjecture}

The starting point of this theory is the paper \cite{mmm}, where Attila
Mar\'oti, Juan Mart\'inez Madrid, and I considered the following question:
for which finite groups $G$ are all Sylow $p$-subgroups needed in order to
cover the set of $p$-elements of $G$? The key observation there is the
following elementary lemma.

\begin{lemma}\label{lem:mmm-picky}
Let $G$ be a finite group and let $p$ be a prime. Then all Sylow
$p$-subgroups are needed in order to cover the set of $p$-elements of $G$
if and only if there exists a $p$-element $x\in G$ lying in a unique
Sylow $p$-subgroup.
\end{lemma}

\begin{proof}
This is \cite[Lemma~2.1]{mmm}.
\end{proof}

This led to the following notion. A $p$-element $x\in G$ is called
\emph{picky} if it lies in a unique Sylow $p$-subgroup of $G$.
For instance, if $G$ has a cyclic Sylow $p$-subgroup $P$, then the generators of $P$ are picky. 

It turns out that picky elements were considered by J. G. Thompson over half a century ago. He conjectured that if $G$ is $p$-solvable and $A$ is a $p$-subgroup of $G$ that is contained in a unique Sylow $p$-subgroup of $G$, then the $p$-length of $G$ is bounded in terms of $|A|$. This was proved by A. Rae \cite{rae} when $A$ is cyclic (i.e., when $A$ is generated by a picky element) using techniques of E. C. Dade and B. Hartley \cite{dad4, har}. He also proved it for arbitrary $A$ when $p$ is odd. As far as I am aware, the conjecture remains open for $p=2$ and $A$ noncyclic. As a consequence of Theorem~B of \cite{rae}, we have that picky elements in $p$-solvable groups have ``large" order. We have observed the same phenomenon in arbitrary finite groups, but it would be interesting to make this precise.

\begin{theorem}[Rae]\label{thm:rae-picky-order}
Let $G$ be a $p$-solvable group of $p$-length $k$. Let $x\in G$ be a picky $p$-element. Then $o(x)\geq p^{k/2}$.
\end{theorem}

As pointed out in \cite{rae}, it might be possible to improve this to $o(x)\geq p^k$.

Another property of picky elements is that they tend to be ``at the top" of the group. For instance, it is not difficult to prove that if $G$ is $p$-solvable then any picky element lies outside $\bO^{p',p}(G)$.

 For a subset
$\mathcal S\subseteq G$, write
$$
\Irr^{\mathcal S}(G)=\{\chi\in\Irr(G)\mid \chi(s)\ne 0
\text{ for some }s\in\mathcal S\}.
$$
For a single element $x$, we write $\Irr^x(G)=\Irr^{\{x\}}(G)$.

Somewhat unexpectedly, \cite{mmm} shows that most finite groups do contain
picky elements. More importantly for the present paper, we noticed in
\cite[Corollary~2.5]{mmm} that those elements are relevant to character
theory.

\begin{proposition}\label{prop:mmm-characters}
Let $x\in G$ be a picky $p$-element. Then $\chi(x)=0$ for every
$\chi\in\Irr(G)$ lying in a $p$-block whose defect groups are not Sylow
$p$-subgroups of $G$.
\end{proposition}

\begin{proof}
This is a reformulation of \cite[Corollary~2.5]{mmm}.
\end{proof}

Thus picky elements are not merely a group-theoretic curiosity: they impose
strong vanishing restrictions on irreducible characters.

We say that a subgroup $Q$ of $G$ is TI if $Q\cap Q^g$ is either $Q$ or $1$
for every $g\in G$. Alperin's Conjecture~C \cite{alp} asserts that, if
$P\in\Syl_p(G)$ is TI and $P^\#=P\setminus\{1\}$, then
$$
|\Irr^{P^\#}(G)|=|\Irr^{P^\#}(\bN_G(P))|.
$$
If $P$ is TI, then every element of $P^\#$ is picky; conversely, if every
element of $P^\#$ is picky, then $P$ is TI. This suggested to me the following
extension of Alperin's Conjecture~C.

\begin{conjecture}\label{conj:picky-cardinality}
Let $G$ be a finite group, let $p$ be a prime, let $P\in\Syl_p(G)$, and let
$\mathcal P$ be a set of picky $p$-elements contained in $P$. Then
$$
|\Irr^{\mathcal P}(G)|=|\Irr^{\mathcal P}(\bN_G(P))|.
$$
\end{conjecture}

The first GAP computations \cite{gap} after formulating Conjecture~\ref{conj:picky-cardinality}
suggested much stronger relations.

\begin{conjecture}[The global picky conjecture]\label{conj:global-picky}
Let $G$ be a finite group, let $p$ be a prime, let $P\in\Syl_p(G)$, and let
$\mathcal P$ be a set of picky $p$-elements contained in $P$. Then there exists
a bijection
$$
f:\Irr^{\mathcal P}(G)\longrightarrow \Irr^{\mathcal P}(\bN_G(P))
$$
such that for every $\chi\in\Irr^{\mathcal P}(G)$:
\begin{enumerate}[label=(\arabic*)]
\item $\chi(1)_p=f(\chi)(1)_p$;
\item $\mathbb Q(\chi(x))=\mathbb Q(f(\chi)(x))$ for every $x\in\mathcal P$;
\item $f(\Irr^x(G))=\Irr^x(\bN_G(P))$ for every $x\in\mathcal P$.
\end{enumerate}
\end{conjecture}

Thus the global picky conjecture contains, simultaneously, Conjecture~\ref{conj:picky-cardinality},
a strong McKay-type condition, and a compatibility statement for each
individual picky element.

\begin{remark}
Since $\Irr_{p'}(G)\subseteq\Irr^x(G)$ for every $p$-element $x$ by
\cite[Corollary~4.20]{navmc}, the (global) picky conjecture at $p$ may be viewed as a
refinement of McKay's conjecture for groups with picky elements.
\end{remark}

Note that Conjecture~A in the Introduction follows from
Conjecture~\ref{conj:global-picky}. Conjecture~A is the ``one element at a
time'' version of Conjecture~\ref{conj:global-picky}. At present, this is the
form on which most of the work has focused.

For a picky element $x\in P$, we say that the strong picky conjecture holds for
$x$ if the bijection in Conjecture~A can be chosen so that
$$
f(\chi)(x)=\pm\chi(x)
$$
for every $\chi\in\Irr^x(G)$. Similarly, in the setting of
Conjecture~\ref{conj:global-picky}, we say that the strong global picky
conjecture holds if the bijection can be chosen so that
$$
f(\chi)(x)=\pm\chi(x)
$$
for every $\chi\in\Irr^{\mathcal P}(G)$ and every $x\in\mathcal P$.

The stronger forms hold in many important families, but not in full generality.
The paradigmatic counterexamples arise from finite simple groups with noncyclic
TI Sylow $p$-subgroups, as we will see in Section~3.

The strong form of Conjecture~A has been verified in several important cases.
Moret\'o, Navarro and Rizo proved the strong version for $p$-solvable groups
when $p$ is odd \cite{mnr}. Malle established Conjecture~A in several cases for
finite groups of Lie type, including regular unipotent elements and various
families in defining characteristic \cite{mal}. Malle and Schaeffer Fry proved
Conjecture~A for all quasi-simple groups of Lie type in non-defining
characteristic, and obtained the strong version under additional natural
conditions \cite{ms}. Finally, Mart\'inez Madrid established Conjecture~A for
symmetric groups, proving the strong version in that setting \cite{mar}.

One of the few cases where the global picky conjecture has been verified in
full generality is for sporadic groups. This was carried out by Breuer, who
determined the pairs $(G,p)$ and elements $x$ for which the strong version
holds \cite{bre}.

\begin{theorem}[Breuer]\label{thm:breuer-sporadic}
Let $G$ be a sporadic simple group and let $p$ be a prime dividing $|G|$. Then
the strong global picky conjecture holds for $G$, except for the pairs $(G,p)$
$$
(HS,5),\ (J_3,3),\ (McL,5),\ (Co_3,5),\ (Co_2,5),\ (J_4,11).
$$
In these cases, the global picky conjecture holds.
\end{theorem}

We conclude this section with an important point about fusion of picky elements.

\begin{lemma}\label{lem:picky-fusion}
Let $G$ be a finite group, let $P\in\Syl_p(G)$, and let $x\in P$ be picky. If
$y\in P$ and $x$ and $y$ are conjugate in $G$, then $x$ and $y$ are conjugate in
$\bN_G(P)$.
\end{lemma}

\begin{proof}
Let $y=x^g$ for some $g\in G$. Then $x\in P$ and $x^g\in P$, so
$x\in P\cap P^{g^{-1}}$. Since $x$ is picky, it belongs to a unique Sylow
$p$-subgroup of $G$. Hence $P=P^{g^{-1}}$, and therefore $g\in\bN_G(P)$. Thus
$x$ and $y$ are conjugate in $\bN_G(P)$.
\end{proof}

This elementary observation is essential for the picky conjectures. It is the
fusion-theoretic reason why, for a picky element $x\in P$, the subgroup
$\bN_G(P)$ is the natural local subgroup to compare with $G$. A first consequence of this lemma is the following.

\begin{corollary}\label{cor:ratreal}
Let $G$ be a finite group, let $P\in\Syl_p(G)$, and let $x\in P$ be picky. Then $x$ is rational (respectively real) as an element of $G$ if and only if $x$ is rational (resp. real) as an element of $\bN_G(P)$.
\end{corollary}

\section{Simple groups with TI Sylow $p$-subgroups}

The following list is \cite[Proposition~1.3]{bm}, which follows from work of
Suzuki when $p=2$ and Gorenstein--Lyons when $p$ is odd.

\begin{theorem}\label{thm:blau-michler-ti}
Let $G$ be a nonabelian simple group with a noncyclic TI Sylow
$p$-subgroup $P$. Then one of the following holds:
\begin{enumerate}
\item $G=\PSL_2(q)$, where $q\geq 4$ and $p$ divides $q$;
\item $G=\PSU_3(q)$, where $q\geq 3$ and $p$ divides $q$;
\item $G={}^2B_2(q)$, where $q=2^{2m+1}>2$ and $p=2$;
\item $G={}^2G_2(q)$, where $q=3^{2m+1}>3$ and $p=3$;
\item $(G,p)$ is one of the groups
$$
        (\PSL_3(4),3),\quad ({}^2F_4(2)',5),\quad
        (M_{11},3),\quad (McL,5),\quad (J_4,11).
$$
\end{enumerate}
\end{theorem}

We shall use this list to prove the following result.

\begin{theorem}\label{thm:ti-sylow-main}
Let $G$ be a finite simple group with a noncyclic TI Sylow $p$-subgroup $P$.
Then the global picky conjecture holds for $G$. Moreover, the strong picky
conjecture holds for every element in $P\setminus P'$ and fails for every
element in $P'\setminus\{1\}$.
\end{theorem}

Since $P$ is TI, every element of $P^\#$ is picky. Hence, by
Lemma~\ref{lem:picky-fusion}, $\bN_G(P)$ controls $G$-fusion in $P$. In the
proof of Theorem~\ref{thm:ti-sylow-main}, we shall call the $\bN_G(P)$-classes
contained in $P\setminus P'$ good classes, and the nonidentity
$\bN_G(P)$-classes contained in $P'$ bad classes.

The isolated cases $\PSL_3(4)$ with $p=3$, ${}^2F_4(2)'$ with $p=5$, and the
sporadic cases $M_{11}$ with $p=3$, $McL$ with $p=5$, and $J_4$ with $p=11$
of Theorem~\ref{thm:ti-sylow-main} have been checked directly with GAP
\cite{gap}. In the cases $\PSL_3(4)$, ${}^2F_4(2)'$, and $(M_{11},3)$, the
Sylow $p$-subgroup is abelian, so there are no bad classes. It remains to
treat the four infinite families in the list, which is done in the following
subsections. Malle proved in \cite[Proposition~4.4]{mal} the picky conjecture for 
$G={}^2G_2(q)$ and $p=3$ (with some refinements). For completeness, we present 
here a complete proof, using the same techniques.

Theorem~\ref{thm:ti-sylow-main}, together with the limited number of
counterexamples to the strong picky conjecture, suggests the following
conjecture.

\begin{conjecture}
Let $G$ be a simple group, let $p$ be a prime, and let $P\in\Syl_p(G)$. If
$x\in P$ is picky and $x\notin P'$, then the strong picky conjecture holds
for $x$.
\end{conjecture}

This formulation may be related to the role of abelian defect groups in
Brou\'e's conjecture. When $P$ is abelian, the conjecture predicts the strong
picky conjecture for all nonidentity picky elements of $P$.

The exceptions to the strong global picky conjecture found by Breuer are also
consistent with this conjecture. For
$$
        (HS,5),\ (J_3,3),\ (McL,5),\ (Co_3,5),\ (Co_2,5),\ (J_4,11),
$$
I checked the strong picky conjecture one class at a time with GAP
\cite{gap}. In each case, the strong picky conjecture fails only for the
picky classes contained in $P'\setminus\{1\}$, and it holds for all picky
classes contained in $P\setminus P'$.

\subsection{$\PSL_2(q)$}

Let $G=\PSL_2(q)$, where $q=p^n>3$, let $P\in\Syl_p(G)$, and put
$H=\bN_G(P)$.

\begin{lemma}\label{lem:psl2-normalizer}
We have
$$
H=P\rtimes C_{(q-1)/(2,q-1)}.
$$
\end{lemma}

\begin{proof}
This follows from the description of Sylow $p$-normalizers in
\cite[Hilfs\-satz~8.21]{hb}.
\end{proof}

\begin{lemma}\label{lem:psl2-even}
Assume that $q$ is even, and let $x\in P^\#$. Then $P^\#$ is one $G$-class and
one $H$-class. Moreover
$$
\Irr^x(G)=\Irr(G)\setminus\{\St\},
\qquad
\Irr^x(H)=\Irr(H),
$$
where $\St$ denotes the Steinberg character of $G$. All characters in these two
sets have odd degree, and their values at $x$ are $\pm1$.
\end{lemma}

\begin{proof}
Since $q$ is even, $G=\SL_2(q)$ and $H=P\rtimes C_{q-1}$ by
Lemma~\ref{lem:psl2-normalizer}. By \cite[Theorem~38.2]{dor}, $P^\#$ is one
$G$-class, the Steinberg character is the only irreducible character of $G$
which vanishes on this class, and all other irreducible characters have value
$\pm1$ on it and odd degree. By Lemma~\ref{lem:picky-fusion}, $P^\#$ is also one
$H$-class.

The group $H$ is a Frobenius group with elementary abelian kernel $P$ and cyclic
complement $C_{q-1}$. Thus $H$ has $q-1$ linear characters and one nonlinear
irreducible character, induced from any nonprincipal irreducible character of $P$. The
linear characters have value $1$ on $P$, and the nonlinear character has value
$-1$ on $P^\#$. All these degrees are odd.
\end{proof}

Assume now that $q$ is odd. 
We will also use the information on the character table of $\SL_2(q)$ obtained from \cite[Theorem~38.1]{dor}.
Let $x_1,x_2\in P^\#$ be representatives of the two
nonidentity $G$-classes contained in $P$. By Lemma~\ref{lem:picky-fusion}, these
are also representatives of the two nonidentity $H$-classes contained in $P$.
Set
$$
\varepsilon=(-1)^{(q-1)/2},
\qquad
\tau=\sqrt{\varepsilon q},
\qquad
a=\frac{-1+\tau}{2},
\qquad
\bar a=\frac{-1-\tau}{2}.
$$

\begin{lemma}\label{lem:psl2-odd-normalizer-characters}
The irreducible characters of $H$ have the following values on both of the classes
represented by $x_1$ and $x_2$:
$$
\begin{array}{c|c|c}
\text{value} & p\text{-part of degree} & \text{multiplicity}\\ \hline
1 & 1 & (q-1)/2\\
a & 1 & 1\\
\bar a & 1 & 1
\end{array}
$$ 
\end{lemma}

\begin{proof}
By Lemma~\ref{lem:psl2-normalizer}, $H=P\rtimes T$, where
$|P|=q$ and $|T|=(q-1)/2$. 
Since the elements of $P^\#$ are picky, Lemma~\ref{lem:picky-fusion} shows that
these are also the two $H$-classes contained in $P^\#$.
Since each $H$-class in $P^\#$ has size at most
$|H:P|=(q-1)/2$, we deduce that both of them have
size $(q-1)/2$. Therefore
$$
        |\bC_H(x_i)|=q
        \qquad (i=1,2).
$$
Since $P\leq \bC_H(x_i)$, it follows that $\bC_H(x_i)=P$ for $i=1,2$.
Thus no element of $H\setminus P$ centralizes a nonidentity element of $P$, and
$H$ is a Frobenius group with kernel $P$.

The characters of $H$ with $P$ in their kernel are precisely the characters of
$H/P$. Thus there are $(q-1)/2$ of them; they are linear and have value $1$ on
both $x_1$ and $x_2$.

It remains to determine the two nonlinear irreducible characters. Since $H$ is
a Frobenius group with abelian kernel $P$, the remaining irreducible characters
are induced from representatives of the $H$-orbits on
$\Irr(P)\setminus\{1_P\}$. The action of $H/P$ on
$\Irr(P)\setminus\{1_P\}$ is also fixed-point-free. Thus all the orbits on
$\Irr(P)\setminus\{1_P\}$ have size $|H:P|=(q-1)/2$, and hence there are
two such orbits. Let the corresponding induced characters be $\beta_1,\beta_2\in\Irr(H)$.
Both have degree $|H:P|=(q-1)/2$ and, in particular, have $p$-part of degree
equal to $1$.

Write $b_i=\beta_i(x_1)$ for $i=1,2$.
By the orthogonality of the columns corresponding to $1$ and $x_1$,
$$
        0=\sum_{\chi\in\Irr(H)}\chi(1)\overline{\chi(x_1)}.
$$
The linear characters contribute $(q-1)/2$, while
$\beta_1(1)=\beta_2(1)=(q-1)/2$. Hence
\[
        b_1+b_2=-1.                         \tag{1}
\]
Also $\bC_H(x_1)=P$, and so the orthogonality of the column corresponding to
$x_1$ gives
\[
        q=\sum_{\chi\in\Irr(H)}|\chi(x_1)|^2
          ={q-1\over 2}+|b_1|^2+|b_2|^2.    \tag{2}
\]

We now use the information on the two unipotent classes already read from the
character table of $G$. If $q\equiv 1\pmod 4$, then both classes are real.
Hence, by Corollary~\ref{cor:ratreal}, $b_1,b_2\in\mathbb R$, and (1),(2) give
$$
        (b_1-b_2)^2=q.
$$
If $q\equiv 3\pmod 4$, then the two classes are inverse to each other. Thus the
values of each character on $x_2$ are the complex conjugates of its values on
$x_1$. Since the columns corresponding to $x_1$ and $x_2$ are orthogonal, we get
$$
        b_1^2+b_2^2=-{q-1\over 2}.
$$
Together with (1), this gives
$$
        b_1b_2={q+1\over 4},
        \qquad
        (b_1-b_2)^2=-q.
$$
Thus in both cases, if
$$
        \varepsilon=(-1)^{(q-1)/2},
$$
we have
$$
        b_1+b_2=-1,
        \qquad
        (b_1-b_2)^2=\varepsilon q.
$$
With the notation fixed before the lemma, the two values on the class of $x_1$
are $a$ and $\bar a$. Renaming $\beta_1,\beta_2$ if necessary, they are
$a,\bar a$ in this order.

The same argument applied to $x_2$ shows that the two nonlinear values on the
class of $x_2$ are again $a$ and $\bar a$. The two columns corresponding to
$x_1$ and $x_2$ cannot be equal, since they represent distinct conjugacy
classes. Since all linear characters have value $1$ on both classes, the two
nonlinear values must therefore be interchanged on the second class. This gives
the displayed tables.
\end{proof}

\begin{lemma}\label{lem:psl2-odd-group-characters}
The irreducible characters of $G$ which are nonzero on the classes represented
by both $x_1$ and $x_2$ have the following values:
$$
\begin{array}{c|c|c}
\text{value} & p\text{-part of degree} & \text{multiplicity}\\ \hline
1 & 1 & m_+\\
-1 & 1 & m_-\\
a & 1 & 1\\
\bar a & 1 & 1
\end{array}
$$
where
$$
m_+=
\begin{cases}
\dfrac{q-1}{4},& \text{if } q\equiv 1\pmod 4,\\[1mm]
\dfrac{q+1}{4},& \text{if } q\equiv 3\pmod 4,
\end{cases}
\qquad
m_-=
\begin{cases}
\dfrac{q-1}{4},& \text{if } q\equiv 1\pmod 4,\\[1mm]
\dfrac{q-3}{4},& \text{if } q\equiv 3\pmod 4.
\end{cases}
$$
\end{lemma}

\begin{proof}
The values are obtained from the character table of $\SL_2(q)$ in
\cite[Theorem~38.1]{dor}, by retaining the irreducible characters with
$\bZ(\SL_2(q))$ in their kernel.

Assume first that $q\equiv 1\pmod 4$. In Dornhoff's notation, the irreducible
characters of $\PSL_2(q)$ are
$$
        1,\quad \St,\quad \eta_1,\quad \eta_2,\quad
        \chi_{2i}\ (1\leq i\leq (q-5)/4),\quad
        \theta_{2j}\ (1\leq j\leq (q-1)/4).
$$
On the two nontrivial unipotent classes, the Steinberg character has value
$0$, the characters $1$ and $\chi_{2i}$ have value $1$, the characters
$\theta_{2j}$ have value $-1$, and the two exceptional characters
$\eta_1,\eta_2$ have values $a,\bar a$, in one order on the class represented
by $x_1$ and in the opposite order on the class represented by $x_2$. Hence
$$
        m_+=1+\frac{q-5}{4}=\frac{q-1}{4},
        \qquad
        m_-=\frac{q-1}{4}.
$$

Assume next that $q\equiv 3\pmod 4$. The irreducible characters of
$\PSL_2(q)$ are
$$
        1,\quad \St,\quad \xi_1,\quad \xi_2,\quad
        \chi_{2i}\ (1\leq i\leq (q-3)/4),\quad
        \theta_{2j}\ (1\leq j\leq (q-3)/4).
$$
On the two nontrivial unipotent classes, the Steinberg character has value
$0$, the characters $1$ and $\chi_{2i}$ have value $1$, the characters
$\theta_{2j}$ have value $-1$, and the two exceptional characters
$\xi_1,\xi_2$ have values $a,\bar a$, in one order on the class represented
by $x_1$ and in the opposite order on the class represented by $x_2$. Hence
$$
        m_+=1+\frac{q-3}{4}=\frac{q+1}{4},
        \qquad
        m_-=\frac{q-3}{4}.
$$
This gives the stated tables. All the nonzero characters appearing there have
degree prime to $p$, and the Steinberg character is the only irreducible
character of $G$ vanishing on the nontrivial unipotent classes.
\end{proof}

\begin{proposition}\label{prop:psl2-picky}
The conclusion of Theorem~\ref{thm:ti-sylow-main} holds for $G=\PSL_2(q)$.
In fact, $G$ satisfies the strong global picky conjecture.
\end{proposition}

\begin{proof}
If $q$ is even, Lemma~\ref{lem:psl2-even} shows that both $\Irr^x(G)$ and
$\Irr^x(H)$ have $q$ characters, all with degree $2$-part equal to $1$, and all
nonzero values on the unique nonidentity class in $P$ are $\pm1$. Choose any
bijection between these two sets. Since there is only one nonidentity class in
$P$, the preceding paragraph shows that it preserves the $2$-parts of degrees
and the values up to sign on that class.

Assume now that $q$ is odd. By Lemmas~\ref{lem:psl2-odd-normalizer-characters}
and~\ref{lem:psl2-odd-group-characters}, the Steinberg character is the only
irreducible character of $G$ which is not in $\Irr^{P^\#}(G)$, and every
irreducible character of $H$ is nonzero on both classes represented by $x_1$ and
$x_2$. Hence
$$
\begin{aligned}
\Irr^{x_1}(G)&=\Irr^{x_2}(G)=\Irr^{P^\#}(G)
              =\Irr(G)\setminus\{\St\},\\
\Irr^{x_1}(H)&=\Irr^{x_2}(H)=\Irr^{P^\#}(H)=\Irr(H).
\end{aligned}
$$
Both sets have cardinality $(q+3)/2$. Choose a bijection by sending the two
exceptional characters of $G$ to the two nonlinear characters of $H$ so that
their values agree on $x_1$, and send all remaining characters bijectively to
the linear characters of $H$. Then their values also agree on $x_2$, by the
displayed tables. All characters involved have degree $p$-part equal to $1$.
For the remaining characters, the values in $G$ are $\pm1$, while the
corresponding linear characters of $H$ have value $1$ on $P$. Thus the values
agree up to sign on both classes. This proves the strong global picky conjecture.
\end{proof}

\subsection{The Suzuki groups ${}^2B_2(q)$}

Let
$$
G={}^2B_2(q)=\Sz(q),\qquad q=2^{2m+1}\geq 8,
\qquad r=\sqrt{2q}=2^{m+1},
$$
let $P\in\Syl_2(G)$, and put $H=\bN_G(P)$. The structure of $H$ is given in
\cite[Theorem~9]{suz}, and the character table of $G$ is given in
\cite[Theorem~13]{suz}.

By \cite[Theorem~13]{suz}, the nonidentity elements of $P$ intersect three
$G$-classes: one class of involutions, represented by $\sigma\in P'=\bZ(P)$,
and two classes of elements of order $4$, represented by $\rho$ and
$\rho^{-1}$ in $P\setminus P'$. By Lemma~\ref{lem:picky-fusion}, these three elements are also representatives of the three classes of $H$ contained in $P^\#$. 

\begin{lemma}\label{lem:suzuki-normalizer-characters}
The irreducible characters of $H$ which are nonzero on the classes represented
by $\sigma,\rho,\rho^{-1}$ give the following data. For the involution class:
$$
\begin{array}{c|c|c}
\text{value} & 2\text{-part of degree} & \text{multiplicity}\\ \hline
1 & 1 & q-1\\
q-1 & 1 & 1\\
-r/2 & r/2 & 2
\end{array}
$$
For each of the two classes of elements of order $4$:
$$
\begin{array}{c|c|c}
\text{value} & 2\text{-part of degree} & \text{multiplicity}\\ \hline
1 & 1 & q-1\\
-1 & 1 & 1\\
ri/2 & r/2 & 1\\
-ri/2 & r/2 & 1
\end{array}
$$
\end{lemma}

\begin{proof}
By Suzuki's description of the normalizer of a Sylow $2$-subgroup
\cite[Theorem~9]{suz}, $H$ is a Frobenius group of order $q^2(q-1)$,
with kernel $P$ and cyclic complement of order $q-1$.  The $H$-class of $\sigma$ is $\bZ(P)\setminus\{1\}$ and has size $q-1$. The other two
classes have the same size, and together they
are $P\setminus P'$. Hence each of them has size $q(q-1)/2$. Thus
\[
        |\bC_H(\sigma)|=q^2,
        \qquad
        |\bC_H(\rho)|=|\bC_H(\rho^{-1})|=2q.        \tag{1}
\]

The characters of $H$ with $P$ in their kernel are exactly the characters of
$H/P$. Thus there are $q-1$ such characters; they are linear and all have value
$1$ on $P$.

We next determine the nonlinear characters. The quotient $H/P'$ is a Frobenius
group with kernel $P/P'$ and cyclic complement of order $q-1$. Since the
complement is transitive on $(P/P')^\#$, this quotient has exactly one nonlinear
irreducible character $\varphi_3\in\Irr(H/P')$ of degree $q-1$. For a Frobenius group with abelian kernel and a transitive complement on the
nonidentity elements of the kernel, this nonlinear character has value $q-1$ on
the identity of $P/P'$ and value $-1$ on the nonidentity elements of $P/P'$.
Therefore
$$
        \varphi_3(\sigma)=q-1,
        \qquad
        \varphi_3(\rho)=\varphi_3(\rho^{-1})=-1.
$$

By the character count in the proof of \cite[Proposition~18]{suz}, $H$ has two
further nonlinear irreducible characters. Denote them by $\varphi_1$ and
$\varphi_2$. Their degrees are
$$
        d=\varphi_1(1)=\varphi_2(1)=\frac{r(q-1)}2 .
$$

We first compute their values on $\sigma$. Put
$s_j=\varphi_j(\sigma)$ for $j=1,2$. Note that since $\sigma$ is an involution, $s_j$ is rational for $j=1,2$. Orthogonality of the columns corresponding
to $1$ and $\sigma$ gives
$$
        0=\sum_{\chi\in\Irr(H)}\chi(1)\overline{\chi(\sigma)}.
$$
The linear characters contribute $q-1$, the character $\varphi_3$ contributes
$(q-1)^2$, and the remaining contribution is
$d(s_1+s_2)$. Hence, since $r^2=2q$,
\[
        s_1+s_2=-r.                                      \tag{2}
\]
Using (1), the orthogonality relation applied to the column of $\sigma$ gives
$$
        q^2=(q-1)+(q-1)^2+s_1^2+s_2^2,
$$
and therefore
\[
        |s_1|^2+|s_2|^2=q.                                \tag{3}
\]
Since $\sigma$ is an involution, the values $s_1,s_2$ are real. From (2) and
(3), using again $r^2=2q$, we get
$$
        s_1=s_2=-r/2.
$$

Now put $a_j=\varphi_j(\rho)$ for $j=1,2$. Orthogonality of the columns
corresponding to $1$ and $\rho$ gives
$$
        0=(q-1)-(q-1)+d(\overline{a_1}+\overline{a_2}),
$$
whence
\[
        a_1+a_2=0.                                      \tag{4}
\]
Using (1), the orthogonality relation applied to  the column of $\rho$ gives
$$
        2q=(q-1)+1+|a_1|^2+|a_2|^2,
$$
and therefore
\[
        |a_1|^2+|a_2|^2=q.                              \tag{5}
\]
Finally, the columns corresponding to $\rho$ and $\rho^{-1}$ are orthogonal.
Since character values on inverse elements are complex conjugates, this gives
$$
        0=(q-1)+1+a_1^2+a_2^2,
$$
that is,
\[
        a_1^2+a_2^2=-q.                                 \tag{6}
\]
By (4), $a_2=-a_1$, and then (6) gives $2a_1^2=-q$. Since $r^2=2q$, after
interchanging $\varphi_1$ and $\varphi_2$ if necessary,
$$
        a_1=ri/2,
        \qquad
        a_2=-ri/2.
$$
The values on $\rho^{-1}$ are the complex conjugates of the values on $\rho$.
This gives the displayed table.
\end{proof}

\begin{lemma}\label{lem:suzuki-group-characters}
The irreducible characters of $G$ which are nonzero on the classes represented
by $\sigma,\rho,\rho^{-1}$ give the following data. For the involution class:
$$
\begin{array}{c|c|c}
\text{value} & 2\text{-part of degree} & \text{multiplicity}\\ \hline
1 & 1 & q/2\\
r-1 & 1 & (q+r)/4\\
-(r+1) & 1 & (q-r)/4\\
-r/2 & r/2 & 2
\end{array}
$$
For each of the two classes of elements of order $4$:
$$
\begin{array}{c|c|c}
\text{value} & 2\text{-part of degree} & \text{multiplicity}\\ \hline
1 & 1 & q/2\\
-1 & 1 & q/2\\
ri/2 & r/2 & 1\\
-ri/2 & r/2 & 1
\end{array}
$$
\end{lemma}

\begin{proof}
By Suzuki's character table \cite[Theorem~13]{suz}, the irreducible characters
of $G$ are
$$
1_G,
\quad X,
\quad X_i,
\quad Y_j,
\quad Z_k,
\quad W_1,
\quad W_2,
$$
where there are
$$
\frac q2-1 \text{ characters } X_i,
\qquad
\frac{q+r}{4} \text{ characters } Y_j,
\qquad
\frac{q-r}{4} \text{ characters } Z_k.
$$
The character $X$ vanishes on all three nonidentity $2$-classes meeting $P$.
The remaining values are
$$
1_G(\sigma)=1_G(\rho)=1_G(\rho^{-1})=1,
$$
$$
X_i(\sigma)=X_i(\rho)=X_i(\rho^{-1})=1,
$$
$$
Y_j(\sigma)=r-1,
\qquad
Y_j(\rho)=Y_j(\rho^{-1})=-1,
$$
$$
Z_k(\sigma)=-(r+1),
\qquad
Z_k(\rho)=Z_k(\rho^{-1})=-1,
$$
$$
W_1(\sigma)=W_2(\sigma)=-r/2,
$$
$$
W_1(\rho)=ri/2,
\qquad
W_2(\rho)=-ri/2,
$$
$$
W_1(\rho^{-1})=-ri/2,
\qquad
W_2(\rho^{-1})=ri/2.
$$
The characters $1_G,X_i,Y_j,Z_k$ have odd degree, while
$$
W_1(1)=W_2(1)=\frac{r(q-1)}2,
$$
whose $2$-part is $r/2$. This gives the displayed tables.
\end{proof}

\begin{proposition}\label{prop:suzuki-picky}
The conclusion of Theorem~\ref{thm:ti-sylow-main} holds for
$G={}^2B_2(q)$.
\end{proposition}

\begin{proof}
By Lemma~\ref{lem:suzuki-group-characters}, $X$ is the only irreducible
character of $G$ which vanishes on the nonidentity $2$-classes meeting $P$.
Thus
$$
\Irr^{P^\#}(G)=\Irr(G)\setminus\{X\}.
$$
Moreover, every character in $\Irr(G)\setminus\{X\}$ is nonzero on each of the
three classes represented by $\sigma,\rho,\rho^{-1}$. By
Lemma~\ref{lem:suzuki-normalizer-characters}, every irreducible character of
$H$ is nonzero on these three classes. Hence
$$
\Irr^x(G)=\Irr^{P^\#}(G),
\qquad
\Irr^x(H)=\Irr^{P^\#}(H)
\qquad (1\neq x\in P).
$$
The two global sets have the same cardinality:
$$
|\Irr^{P^\#}(G)|
=1+\left(\frac q2-1\right)+\frac{q+r}{4}+\frac{q-r}{4}+2
=q+2
=|\Irr(H)|.
$$
Choose a bijection by sending $W_1,W_2$ to $\varphi_1,\varphi_2$ in the order
displayed in Lemma~\ref{lem:suzuki-normalizer-characters}, sending one of the
odd-degree characters with value $-1$ on $\rho$ to $\varphi_3$, and sending the
remaining odd-degree characters bijectively to the linear characters of $H$.
This preserves the $2$-parts of degrees. It also preserves the fields of values
on all nonidentity elements of $P$: the values of odd-degree characters are
rational, and the only nonrational values are $\pm ri/2$, which occur on the two
characters with degree $2$-part $r/2$ in both tables.

On the good classes $\rho$ and $\rho^{-1}$, the same bijection preserves values
up to sign. Indeed, the values of $W_1,W_2$ and $\varphi_1,\varphi_2$ agree on
$\rho$ and on $\rho^{-1}$ by the displayed tables, one odd-degree character with
value $-1$ is sent to $\varphi_3$, and the remaining odd-degree characters are
sent to linear characters, whose value on $P$ is $1$.

On the bad class represented by the involution $\sigma$, strong picky fails.
The character $\varphi_3\in\Irr(H) 
$ has odd degree and value $q-1$ on $\sigma$. For $G$, the possible values of odd-degree characters on $\sigma$ are
$$
1,
\qquad
r-1,
\qquad
-(r+1).
$$
Since $q=2^{2m+1}\geq 8$ and $r^2=2q$, we have $q-1>r+1>1$. Thus $q-1$ is not
equal, even up to sign, to any of these values. Hence no bijection can preserve
both the $2$-parts of degrees and values up to sign on the involution class.
This proves the proposition.
\end{proof}

\subsection{$\PSU_3(q)$}

Let $G=\PSU_3(q)$, where $q$ is a power of the defining prime $p$. Put
$$
d=(3,q+1),
\qquad
r'=\frac{q+1}{d},
$$
and let $P\in\Syl_p(G)$ and $H=\bN_G(P)$.
By \cite[Table~8.5]{bhrd}, $\SU_3(q)$ has a maximal subgroup with structure
$E_q^{1+2}:(q^2-1)$ in the class $\mathcal C_1$, where $E_q^{1+2}$ denotes a
$p$-group of order $q^3$ whose centre and derived subgroup are equal,
elementary abelian of order $q$, and whose quotient by the centre is elementary
abelian of order $q^2$. Since $d=|\bZ(\SU_3(q))|=(q+1,3)$, it follows that
$$
H=P\rtimes C_{(q^2-1)/d}.
$$
Also $P$ is semi-extraspecial (see \cite{fam} for the definition, for
instance), $P'=\bZ(P)$, $|P'|=q$, and $P/P'$ is elementary abelian of order
$q^2$.

By \cite[Table~2]{sf}, the nonidentity $p$-elements of $P$ split into $d$
classes in $P\setminus P'$ and one class in $P'\setminus\{1\}$. By
Lemma~\ref{lem:picky-fusion}, these are also the classes in $H$ contained in
$P^\#$. Let $u\in P\setminus P'$ be in a good class, and let
$z\in P'\setminus\{1\}$ be in the bad class.

\begin{lemma}\label{lem:psu3-normalizer-characters}
The following tables give, for $H=\bN_G(P)$, the values, $p$-parts of degrees,
and multiplicities of the irreducible characters which are nonzero on the good
and bad classes.

If $d=1$, then for a good class the data are
$$
\begin{array}{c|c|c}
\text{value} & p\text{-part of degree} & \text{multiplicity}\\ \hline
1 & 1 & q^2-1\\
-1 & 1 & 1
\end{array}
$$
and for the bad class the data are
$$
\begin{array}{c|c|c}
\text{value} & p\text{-part of degree} & \text{multiplicity}\\ \hline
1 & 1 & q^2-1\\
q^2-1 & 1 & 1\\
-q & q & q+1
\end{array}
$$
If $d=3$, then for a good class the data are
$$
\begin{array}{c|c|c}
\text{value} & p\text{-part of degree} & \text{multiplicity}\\ \hline
1 & 1 & (q^2-1)/3\\
q-r' & 1 & 1\\
-r' & 1 & 2
\end{array}
$$
and for the bad class the data are
$$
\begin{array}{c|c|c}
\text{value} & p\text{-part of degree} & \text{multiplicity}\\ \hline
1 & 1 & (q^2-1)/3\\
(q^2-1)/3 & 1 & 3\\
-q & q & r'
\end{array}
$$
\end{lemma}

\begin{proof}
Put
$$
        e=\frac{q^2-1}{d}.
$$
The characters of $H$ with $P$ in their kernel are the characters in
$\Irr(H/P)$. Thus there are $e$ such characters, all linear, and they all have
value $1$ on every element of $P$.

Next consider the characters of $H$ with $P'$ in their kernel but not $P$ in
their kernel. Put $V=P/P'$. Then $H/P'$ has the form
$$
        V\rtimes C_e,
$$
where $V$ is elementary abelian of order $q^2$. The complement $C_e=H/P$ acts
fixed-point-freely on $V^\#$. Its orbits on $V\setminus\{1\}$ have length
$e$, and there are $d$ of them. The same is true for the action on
$\Irr(V)\setminus\{1_V\}$. Hence the characters of $H/P'$ lying over
nonprincipal characters of $V$ are obtained by inducing representatives of the
$d$ orbits in $\Irr(V)\setminus\{1_V\}$, and all have degree $e$.

Let $\mu\in\Irr(V)\setminus\{1_V\}$, and let $\chi_\mu=\mu^{H/P'}$. If
$v\in V\setminus\{1\}$, then
$$
        \chi_\mu(v)=\sum\nu(v),
$$
where $\mu$ runs over the $C_e$-orbit of $\mu$. 
If $d=1$, the $C_e$-orbit of 
 $\mu$ is all of $\Irr(V)\setminus\{1_V\}$, and
hence
$$
        \chi_\mu(v)=\sum_{\nu\in\Irr(V)\setminus\{1_V\}}\nu(v)=-1.
$$
This gives the second row in the good-class table for $d=1$.

Assume now that $d=3$. We identify $V$ with the additive group of
$\mathbb F_{q^2}$ in such a way that $C_e$ is a subgroup of
$\mathbb F_{q^2}^{\times}$ of index $3$. Since $3$ divides $q+1$, the subgroup
$\mathbb F_q^\times$ is contained in $C_e$. Thus each $C_e$-orbit on
$\Irr(V)\setminus\{1_V\}$ is a union of
$$
        r'=\frac{q+1}{3}
$$
one-dimensional $\mathbb F_q$-subspaces of $\Irr(V)$.

Fix $v\in V\setminus\{1\}$. The characters in $\Irr(V)$ may be viewed as the
additive characters of the $\mathbb F_q$-vector space $V$. The characters which
are trivial on the subgroup $\langle v\rangle_{\mathbb F_q}=\{av\mid a\in{\mathbb F_q}\}$ form a unique
one-dimensional $\mathbb F_q$-subspace of $\Irr(V)$. If $L$ is a
one-dimensional $\mathbb F_q$-subspace of $\Irr(V)$, then
$$
        \sum_{\nu\in L\setminus\{1_V\}}\nu(v)
        =
        \begin{cases}
        q-1, & \text{if all characters in } L \text{ are trivial on }
                    \langle v\rangle_{\mathbb F_q},\\
        -1, & \text{otherwise.}
        \end{cases}
$$
Therefore, among the three $C_e$-orbits on $\Irr(V)\setminus\{1_V\}$, one gives
the value
$$
        (q-1)+(r'-1)(-1)=q-r',
$$
and the other two give the value
$$
        r'(-1)=-r'.
$$
This gives the second and third rows in the good-class table for $d=3$.

Now let $z\in P'\setminus\{1\}$. Since the characters just considered have
$P'$ in their kernel, their value on $z$ is their degree. Hence, on the bad
class, they contribute the value $q^2-1$ with multiplicity $1$ when $d=1$, and
the value $(q^2-1)/3$ with multiplicity $3$ when $d=3$. Together with the
linear characters, this gives all rows of $p'$-degree in the bad-class tables.

It remains to consider the characters lying over nonprincipal characters of
$P'$. By Theorem~A of \cite{fam}, since $P$ is semi-extraspecial,
$$
        \cd(P)=\{1,q\}.
$$
Let $\lambda\in\Irr(P')\setminus\{1_{P'}\}$. Since $P'$ is central in $P$, every irreducible constituent of $\lambda^P$ lies
over $\lambda$. Moreover $\lambda^P$ has degree $|P:P'|=q^2$ and vanishes on
$P\setminus P'$. If
$\varphi\in\Irr(P)$ lies over $\lambda$, then
$\varphi_{P'}=\varphi(1)\lambda=q\lambda$. Hence the multiplicity of
$\varphi$ in $\lambda^P$ is $q$. It follows from the degree of $\lambda^P$
that there is a unique such character, say $\varphi_\lambda$, and
$$
        \lambda^P=q\varphi_\lambda,
        \qquad
        \varphi_\lambda(1)=q.
$$
In particular, $\varphi_\lambda$ vanishes on $P\setminus P'$.

The group $H/P$ is transitive on $P'\setminus\{1\}$, again by the class
information and Lemma~\ref{lem:picky-fusion}. Thus the inertia subgroup of
$\varphi_\lambda$ in $H$ is $I=P\rtimes T_\lambda$, where
$$
        |T_\lambda|=\frac{|H/P|}{q-1}=\frac{q+1}{d}=r'.
$$
Since $I/P$ is cyclic and $(\varphi_\lambda(1),|I:P|)=1$, the character
$\varphi_\lambda$ extends to $I$ by \cite[Problem~6.17]{isa}. If
$\widetilde\varphi_\lambda$ is such an extension and $\tau\in\Irr(I/P)$, then
Clifford's theorem \cite[Theorem~6.11]{isa} gives that
$(\widetilde\varphi_\lambda\tau)^H$
is irreducible. Varying $\tau$, we obtain exactly $r'$ irreducible characters
of $H$, all of degree
$$
        q|H:I|=q(q-1),
$$
whose $p$-part is $q$.

These characters vanish on $P\setminus P'$. Let $z\in P'\setminus\{1\}$, and
let
$$
        \psi=(\widetilde\varphi_\lambda\tau)^H.
$$
Since $(\varphi_\lambda)_{P'}=q\lambda$, the induction formula gives
$$
\begin{aligned}
\psi(z)
&=\frac{|P|}{|I|}\sum_{t\in H/P}\varphi_\lambda(z^t)  \\
&=\frac{q}{|T_\lambda|}\sum_{t\in H/P}\lambda(z^t).
\end{aligned}
$$
The orbit of $z$ under $H/P$ is $P'\setminus\{1\}$, and each element occurs
with multiplicity $|T_\lambda|$. Since $\lambda$ is nonprincipal,
$$
\sum_{t\in H/P}\lambda(z^t)
=|T_\lambda|\sum_{y\in P'\setminus\{1\}}\lambda(y)
=-|T_\lambda|.
$$
Therefore $\psi(z)=-q$. This gives the rows of degree $p$-part $q$ in the
bad-class tables.
\end{proof}

\begin{lemma}\label{lem:psu3-group-characters}
For $G=\PSU_3(q)$, the values, $p$-parts of degrees, and multiplicities of the
irreducible characters which are nonzero on the good and bad classes are as
follows.

If $d=1$, then for a good class the data are
$$
\begin{array}{c|c|c}
\text{value} & p\text{-part of degree} & \text{multiplicity}\\ \hline
1 & 1 & q(q+1)/2\\
-1 & 1 & q(q-1)/2
\end{array}
$$
and for the bad class the data are
$$
\begin{array}{c|c|c}
\text{value} & p\text{-part of degree} & \text{multiplicity}\\ \hline
1 & 1 & q(q-1)/2\\
-(q-1) & 1 & q\\
2q-1 & 1 & q(q-1)/6\\
-(q+1) & 1 & q(q-1)/3\\
-q & q & 1\\
q & q & q
\end{array}
$$
If $d=3$, then for a good class the data are
$$
\begin{array}{c|c|c}
\text{value} & p\text{-part of degree} & \text{multiplicity}\\ \hline
1 & 1 & q(q+1)/6\\
-1 & 1 & (q-2)(q+1)/6\\
q-r' & 1 & 1\\
-r' & 1 & 2
\end{array}
$$
and for the bad class the data are
$$
\begin{array}{c|c|c}
\text{value} & p\text{-part of degree} & \text{multiplicity}\\ \hline
1 & 1 & (q^2-q+4)/6\\
-(q-1) & 1 & (q-2)/3\\
2r'-1 & 1 & 3\\
2q-1 & 1 & (q-2)(q+1)/18\\
-(q+1) & 1 & (q-2)(q+1)/9\\
-q & q & 1\\
q & q & (q-2)/3
\end{array}
$$
\end{lemma}

\begin{proof}
We use \cite[Table~2]{sf}. In the notation of Simpson and Frame, in the unitary
case one has
$$
r=q+1,
\qquad
s=q-1,
\qquad
t=q^2-q+1,
\qquad
r'=\frac{q+1}{d},
\qquad
t'=\frac{t}{d},
$$
and
$$
r''=0,
\qquad
t''=\frac{t'-1}{6},
\qquad
d'=\frac{3-d}{2}.
$$
The class $C_2$ is the bad class and the classes $C_3^{(\ell)}$ are the good
classes.

One entry in \cite[Table~2]{sf} appears to be a misprint: the entry in the
column $X_{q^3}$ and row $C_2$ is printed as $\delta q$. This column is the
Steinberg character, so its value on the nontrivial unipotent class $C_2$
should be $0$. We use
$$
X_{q^3}(C_2)=0.
$$

In the following lists, a pair attached to a character $\chi$ means
$(\chi(x),\chi(1)_p)$, where $x$ is an element in the class under
consideration.

For a good class, the contributing characters in \cite[Table~2]{sf} give:
$$
(1,1) \text{ with multiplicity } r'+3t''-d',
$$
$$
(-1,1) \text{ with multiplicity } 3t'',
$$
together with the contribution of the characters denoted $X_{s'r'}$:
$$
(q-r',1) \text{ once},
\qquad
(-r',1) \text{ with multiplicity } d-d'-1.
$$
If $d=1$, then $d'=1$, $r'=q+1$, and $t''=q(q-1)/6$; the family $X_{s'r'}$
does not occur. Thus the multiplicities are
$$
r'+3t''-d'=\frac{q(q+1)}{2},
\qquad
3t''=\frac{q(q-1)}{2}.
$$
If $d=3$, then $d'=0$ and $t''=(q-2)(q+1)/18$, giving
$$
r'+3t''=\frac{q(q+1)}{6},
\qquad
3t''=\frac{(q-2)(q+1)}{6},
$$
and $X_{s'r'}$ contributes $(q-r',1)$ once and $(-r',1)$ twice. This gives the
tables for good classes.

For the bad class, the characters of $p'$-degree in \cite[Table~2]{sf} give:
$$
(1,1) \text{ with multiplicity } 1+3t''-d',
$$
$$
(-(q-1),1) \text{ with multiplicity } r'-1,
$$
$$
(2r'-1,1) \text{ with multiplicity } d-d',
$$
$$
(2q-1,1) \text{ with multiplicity } t'',
$$
$$
(-(q+1),1) \text{ with multiplicity } 2t''.
$$
The remaining characters nonzero on the bad class are $X_{qs}$ and the
characters $X_{qt}^{(u)}$; they contribute
$$
(-q,q) \text{ once},
\qquad
(q,q) \text{ with multiplicity } r'-1.
$$
Substituting $d=1$ and $d=3$ gives exactly the two bad-class tables above.
\end{proof}

\begin{proposition}\label{prop:psu3-ti}
The conclusion of Theorem~\ref{thm:ti-sylow-main} holds for $G=\PSU_3(q)$.
\end{proposition}

\begin{proof}
For a good class, Lemmas~\ref{lem:psu3-normalizer-characters} and
\ref{lem:psu3-group-characters} show that the displayed pairs agree up to sign.
Thus the strong picky conjecture holds for the elements of $P\setminus P'$.

For the bad class, compare the same two lemmas. The characters of $p'$-degree
can be matched preserving the $p$-parts of degrees and fields of values, since
all the displayed values are rational and the total multiplicities agree. The
characters with $p$-part of degree equal to $q$ have values
$$
        -q \text{ once},
        \qquad
        q \text{ with multiplicity } r'-1
$$
for $G$, and value $-q$ with multiplicity $r'$ for $H$. Thus the fields of
values and the $p$-parts of degrees are preserved, and the global picky
conjecture holds.

The strong picky conjecture fails on the bad class.  For instance, the value
$-(q-1)$ occurs for $G$, whereas no value with absolute value $q-1$ occurs for
$H$.
\end{proof}

\subsection{The small Ree groups ${}^2G_2(q)$}

Let $G={}^2G_2(q)$, where
$$
        q=3^{2k+1}=3m^2,
        \qquad m=3^k,
        \qquad q>3,
$$
and let $P\in \Syl_3(G)$. Put $H=\bN_G(P)$.

We shall use Ward's notation for the classes of $3$-elements in $H$, as in
his character table for $G$ \cite[pp.~87--88]{war}. Thus the nonidentity
$3$-elements contained in $P$ are represented by the six classes
$$
        X\in \bZ(P)\setminus\{1\},
        \qquad
        T,\ T^{-1}\in P'\setminus \bZ(P),
        \qquad
        Y,\ YT,\ YT^{-1}\in P\setminus P'.
$$
Ward proves that $P$ has order $q^3$, that $\bZ(P)$ is elementary abelian of
order $q$, and that $P$ contains a normal elementary abelian subgroup
$P'=\Phi(P)$ of order $q^2$ containing $\bZ(P)$; moreover
$$
        H=PW,
        \qquad W\cong C_{q-1}
$$
\cite[p.~63]{war}.

For the normalizer $H$, unlike in the preceding families, we shall use directly
the character table computed by van der Waall for the group $\mathfrak R_\sigma$
isomorphic to $H$ \cite[Section~2]{vdw}. Van der Waall uses a different
notation for the conjugacy classes. Structurally, the correspondence is
$$
\begin{array}{c|c}
\text{Ward} & \text{van der Waall} \\ \hline
X & (1,0,0,1)\\
\{T,T^{-1}\} & \{(1,0,1,0),(1,0,-1,0)\}\\
\{Y,YT,YT^{-1}\} &
\{(1,1,0,0),(1,1,x,0),(1,1,-x,0)\}.
\end{array}
$$
We choose the representatives so that the values in van der Waall's table
agree with the values in Ward's notation used below; with this convention,
$$
\begin{gathered}
T=(1,0,1,0),\quad T^{-1}=(1,0,-1,0),\quad Y=(1,1,0,0),\\
YT=(1,1,x,0),\quad YT^{-1}=(1,1,-x,0).
\end{gathered}
$$
We shall use Ward's notation throughout.

We shall use Ward's character table for $G$ \cite[pp.~87--88]{war} and van der
Waall's character table for $H$ \cite[Section~2]{vdw}. Since $P$ is TI, every
element of $P^\#$ is picky. Hence, by Lemma~\ref{lem:picky-fusion}, the
$G$-classes and the $H$-classes contained in $P^\#$ coincide.

\begin{lemma}\label{lem:ree-normalizer-values}
Let $H=\bN_G(P)$. The following tables give the values, $3$-parts of degrees,
and multiplicities of the irreducible characters of $H$ which are nonzero on
the classes represented by $X$, $T$, $T^{-1}$, $Y$, $YT$ and $YT^{-1}$.

For the class $X$:
$$
\begin{array}{c|c|c}
\text{value} & 3\text{-part of degree} & \text{multiplicity}\\ \hline
1 & 1 & q-1\\
q-1 & 1 & 1\\
m(q-1) & m & 2\\
m(q-1)/2 & m & 4\\
-q & q & 1
\end{array}
$$
For the class $T$, put
$$
b=-\frac12+\frac{i\,m\sqrt 3}{2}.
$$
Then:
$$
\begin{array}{c|c|c}
\text{value} & 3\text{-part of degree} & \text{multiplicity}\\ \hline
1 & 1 & q-1\\
q-1 & 1 & 1\\
2mb & m & 1\\
2m\bar b & m & 1\\
mb & m & 2\\
m\bar b & m & 2
\end{array}
$$
and the table for $T^{-1}$ is obtained by complex conjugation. For the class
$Y$:
$$
\begin{array}{c|c|c}
\text{value} & 3\text{-part of degree} & \text{multiplicity}\\ \hline
1 & 1 & q-1\\
-1 & 1 & 1\\
m & m & 4\\
-m & m & 2
\end{array}
$$
Finally, put
$$
a=-\frac12+\frac{i\sqrt 3}{2}.
$$
For the class $YT$:
$$
\begin{array}{c|c|c}
\text{value} & 3\text{-part of degree} & \text{multiplicity}\\ \hline
1 & 1 & q-1\\
-1 & 1 & 1\\
m\bar a & m & 2\\
ma & m & 2\\
-ma & m & 1\\
-m\bar a & m & 1
\end{array}
$$
and the table for $YT^{-1}$ is obtained by complex conjugation.
\end{lemma}

\begin{proof}
We read the character values from van der Waall's character table for the group
$\mathfrak R_\sigma$, which is isomorphic to $H$, translating the class notation
as fixed above. The degrees of the characters used here are read from
\cite[Theorem~10, p.~165, and p.~170]{vdw}.

The $q-1$ linear characters have $P$ in their kernel, and hence have value $1$
on all elements of $P$. The character denoted by $\rho$ in van der Waall's table
has degree $q-1$; it has value $q-1$ on $X$, $T$ and $T^{-1}$, and value $-1$
on $Y$, $YT$ and $YT^{-1}$.

The characters $\psi_1,\psi_2$ have degree $m(q-1)$, and the characters
$\psi_3,\psi_4,\psi_5,\psi_6$ have degree $m(q-1)/2$. These degrees have
$3$-part equal to $m$. Their values on the six classes are those displayed in
van der Waall's table, after translating his notation into Ward's notation.
Finally, the character denoted by $\psi$ has degree $q(q-1)$, whose $3$-part is
$q$; it has value $-q$ on the class $X$ and value $0$ on the remaining classes
in $P^\#$. This gives the displayed tables.
\end{proof}

\begin{lemma}\label{lem:ree-group-values}
For $G={}^2G_2(q)$, the irreducible characters which are nonzero on the classes
represented by $X$, $T$, $T^{-1}$, $Y$, $YT$ and $YT^{-1}$ have the following
values, $3$-parts of degrees and multiplicities.

For the class $X$:
$$
\begin{array}{c|c|c}
\text{value} & 3\text{-part of degree} & \text{multiplicity}\\ \hline
1 & 1 & (q-1)/2\\
-(q-1) & 1 & 1\\
2q-1 & 1 & (q-5)/2\\
-(q+1+3m) & 1 & 1\\
-(q+1-3m) & 1 & 1\\
-(q+m)/2 & m & 2\\
(q-m)/2 & m & 2\\
-m & m & 2\\
q & q & 1
\end{array}
$$
For the class $T$:
$$
\begin{array}{c|c|c}
\text{value} & 3\text{-part of degree} & \text{multiplicity}\\ \hline
1 & 1 & (q+1)/2\\
-1 & 1 & (q-5)/2\\
-3m-1 & 1 & 1\\
3m-1 & 1 & 1\\
(-m+im^2\sqrt 3)/2 & m & 2\\
(-m-im^2\sqrt 3)/2 & m & 2\\
-m+im^2\sqrt 3 & m & 1\\
-m-im^2\sqrt 3 & m & 1
\end{array}
$$
and the table for $T^{-1}$ is obtained by complex conjugation. For the class
$Y$:
$$
\begin{array}{c|c|c}
\text{value} & 3\text{-part of degree} & \text{multiplicity}\\ \hline
1 & 1 & (q+1)/2\\
-1 & 1 & (q-1)/2\\
m & m & 4\\
-m & m & 2
\end{array}
$$
For the class $YT$:
$$
\begin{array}{c|c|c}
\text{value} & 3\text{-part of degree} & \text{multiplicity}\\ \hline
1 & 1 & (q+1)/2\\
-1 & 1 & (q-1)/2\\
m\bar a & m & 2\\
ma & m & 2\\
-ma & m & 1\\
-m\bar a & m & 1
\end{array}
$$
and the table for $YT^{-1}$ is obtained by complex conjugation.
\end{lemma}

\begin{proof}
The values are read from Ward's character table \cite[pp.~87--88]{war}. The
values on $T^{-1}$ and $YT^{-1}$ are the complex conjugates of the values on
$T$ and $YT$, respectively. The $3$-parts of the degrees are obtained from
Ward's degree formulas in the same table: $\xi_4(1)_3=q$, the characters
$\xi_5,\ldots,\xi_{10}$ have degree $3$-part equal to $m$, and the remaining
characters displayed here have degree prime to $3$. The multiplicities are
read directly from Ward's table: the rows $\eta_r,\eta'_r$ represent families
of $(q-3)/4$ characters, while $\eta_t,\eta'_t$ represent families of
$(q-5)/4$ characters, and $\eta_i^-$ and $\eta_i^+$ are individual characters.
\end{proof}

\begin{proposition}\label{prop:ree}
The conclusion of Theorem~\ref{thm:ti-sylow-main} holds for
$G={}^2G_2(q)$.
\end{proposition}

\begin{proof}
Let $\mathcal P=P^\#$. Ward's table shows that the only irreducible character
of $G$ vanishing on all nonidentity $3$-elements is $\xi_3$, while $\xi_4$ is
nonzero on $X$ and vanishes on every noncentral $3$-class. Hence
$$
\Irr^{\mathcal P}(G)=\Irr(G)\setminus\{\xi_3\},
\qquad
\Irr^X(G)=\Irr^{\mathcal P}(G),
$$
and, for
$$
x\in\{T,T^{-1},Y,YT,YT^{-1}\},
$$
one has
$$
\Irr^x(G)=\Irr(G)\setminus\{\xi_3,\xi_4\}.
$$
For $H$, Lemma~\ref{lem:ree-normalizer-values} gives
$$
\Irr^{\mathcal P}(H)=\Irr(H),
\qquad
\Irr^X(H)=\Irr(H),
$$
and
$$
\Irr^x(H)=\Irr(H)\setminus\{\psi\}
$$
for every noncentral nonidentity $3$-element $x\in P$.

Define a bijection
$$
f:\Irr^{\mathcal P}(G)\longrightarrow \Irr^{\mathcal P}(H)
$$
by sending
$$
\xi_4\mapsto \psi,
$$
and
$$
\xi_5\mapsto \psi_3,
\quad
\xi_6\mapsto \psi_4,
\quad
\xi_7\mapsto \psi_5,
\quad
\xi_8\mapsto \psi_6,
\quad
\xi_9\mapsto \psi_2,
\quad
\xi_{10}\mapsto \psi_1.
$$
The remaining $q$ characters, all of degree prime to $3$, are mapped
bijectively to the $q-1$ linear characters of $H$ together with $\rho$.
For the bad classes $X$, $T$ and $T^{-1}$, this bijection need not preserve
values up to sign. The displayed tables show, however, that it preserves the
$3$-parts of degrees and the fields of values on these classes. For the good
classes $Y$, $YT$ and $YT^{-1}$, the same tables show that the bijection
preserves values up to sign. In all cases it maps $\Irr^x(G)$ onto
$\Irr^x(H)$. Hence the global picky conjecture holds, and the strong picky
conjecture holds for all elements of $P\setminus P'$.

It remains to prove failure of the strong form on the bad classes. For $X$, the
only values of characters of $H$ of degree prime to $3$ are
$$
1\qquad\text{and}\qquad q-1,
$$
whereas characters of $G$ of degree prime to $3$ take the value $2q-1$ on $X$.
Since
$$
        2q-1\neq \pm1,
        \qquad
        2q-1\neq \pm(q-1),
$$
there is no bijection preserving the $3$-parts of degrees and values up to sign
on $X$.

For $T$, the only values of characters of $H$ of degree prime to $3$ are again
$$
1\qquad\text{and}\qquad q-1.
$$
But the characters $\eta_i^-$ and $\eta_i^+$ of $G$ take the values
$$
        -3m-1
        \qquad\text{and}\qquad
        3m-1
$$
on $T$. Since $m\geq 3$, neither $3m-1$ nor $3m+1$ is equal to $1$ or to
$q-1=3m^2-1$. Thus no bijection can preserve the $3$-parts of degrees and
values up to sign on $T$. The same argument, after complex conjugation, gives
the failure for $T^{-1}$.
\end{proof}

\subsection{Malle's refinement}

In private communication \cite{mal24}, Gunter Malle suggested that one might hope for a further
refinement of the picky and subnormalizer conjectures, namely that the $p$-parts of
the character values should also match under the expected correspondence. To state this, recall that if
$0\ne\alpha\in \overline{\mathbb C}$ is an algebraic integer, then one defines its
$p$-part by
$$
\alpha_p:=\bigl|N_{\mathbb Q(\alpha)/\mathbb Q}(\alpha)\bigr|_p^{1/[\mathbb Q(\alpha):\mathbb Q]}.
$$
Here $N_{\mathbb Q(\alpha)/\mathbb Q}$ denotes the field norm. When $\alpha\in \mathbb Z$, this is just the usual $p$-part of the integer $\alpha$.

Malle asked whether, in addition to conditions (1) and (2) in Conjectures~A
and~B, one can also require
\begin{enumerate}
\item[(3)] $\chi(x)_p=f(\chi)(x)_p$ for every $\chi$.
\end{enumerate}

I expect the bijections in Conjectures~A and~B to satisfy this additional
condition. Indeed, whenever the strong picky conjecture holds, condition~(3) is
automatic. A careful analysis of the tables in the preceding subsections shows that condition~(3) also holds for the bad classes in the groups above.

\section{$p$-sections and reduced $p$-sections}

We now consider stronger versions of the subnormalizer conjecture.

For a subgroup $A$ of $G$, let
$$
        S_G(A)=\{g\in G\mid A\trianglelefteq\trianglelefteq \langle A,g\rangle\}
$$
be the subnormalizer set of $A$ in the sense of Casolo, and put
$$
        \Sub_G(A)=\langle S_G(A)\rangle.
$$
If $A=\langle x\rangle$ is cyclic, we write $S_G(x)$ and $\Sub_G(x)$.
If $P\in\Syl_p(G)$ and $x$ is a $p$-element of $G$, let
$$
        \lambda_G(x)=|\{\,Q\in\Syl_p(G)\mid x\in Q\,\}|.
$$

\begin{lemma}\label{lem:subnormalizer-picky}
Let $G$ be a finite group, let $P\in \Syl_p(G)$, and let $x\in P$. Then
$\bN_G(P)\subseteq \Sub_G(x)$. Moreover, $\Sub_G(x)=\bN_G(P)$ if and only if
$x$ is picky.
\end{lemma}

\begin{proof}
We first prove that $\bN_G(P)\subseteq S_G(x)$. Let $g\in \bN_G(P)$. The
subgroup generated by the conjugates of $\langle x\rangle$ under
$\langle x,g\rangle$ is a $p$-subgroup of $P$ and is normal in
$\langle x,g\rangle$. Since every subgroup of a finite $p$-group is
subnormal, it follows that $\langle x\rangle$ is subnormal in
$\langle x,g\rangle$. Hence $g\in S_G(x)$.

Casolo's theorem on subnormalizers \cite{cas} gives
$$
        |S_G(x)|=\lambda_G(x)|\bN_G(P)|.
$$
If $x$ is picky, then $\lambda_G(x)=1$. Since $\bN_G(P)\subseteq S_G(x)$, this
forces $S_G(x)=\bN_G(P)$, and therefore $\Sub_G(x)=\bN_G(P)$. Conversely, if
$\Sub_G(x)=\bN_G(P)$, then $S_G(x)=\bN_G(P)$, and Casolo's formula gives
$\lambda_G(x)=1$. Thus $x$ is picky.
\end{proof}

Put $N=\bN_G(P)$. The permutation character $(1_N)^G$ is the character of
the action of $G$ by conjugation on the set $\Syl_p(G)$. Hence, for $x\in P$,
the value $(1_N)^G(x)$ is the number of Sylow $p$-subgroups fixed by $x$, or
equivalently normalized by $x$. Since $x$ is a $p$-element, a Sylow
$p$-subgroup normalized by $x$ contains $x$. Therefore
$$
        \lambda_G(x)=(1_N)^G(x).
$$
This identity was also noticed in \cite{kn} and \cite{sam}. Consequently, for
$x\in P$,
$$
        x \text{ is picky}
        \quad\Longleftrightarrow\quad
        \lambda_G(x)=1
        \quad\Longleftrightarrow\quad
        (1_N)^G(x)=1.
$$

The following result generalizes Lemma~\ref{lem:picky-fusion}.

\begin{lemma}\label{lem:subnormalizer-fusion}
Let $G$ be a finite group, let $x$ be a $p$-element of $G$, and let
$y\in\Sub_G(x)$. If $x$ and $y$ are conjugate in $G$, then they are conjugate
in $\Sub_G(x)$.
\end{lemma}

\begin{proof}
Put $R=\Sub_G(x)$, and choose $P\in\Syl_p(G)$ with $x\in P$. By
Lemma~\ref{lem:subnormalizer-picky}, we have $\bN_G(P)\leq R$. In particular
$P\leq R$, and since $P\in\Syl_p(G)$, it follows that $P\in\Syl_p(R)$.
Since $y\in R$ is a $p$-element, there exists $r\in R$ such that $y^r\in P$.
Thus it is enough to prove that, if $a,b\in P$ are conjugate in $G$, then they
are conjugate in $\Sub_G(a)$.

Let $a,b\in P$ and suppose that $a^g=b$ for some $g\in G$. We use Alperin's
Fusion Theorem in the form stated in \cite[Theorem~1.29]{cra}. Thus the pairs
$(Q,X)$, where $Q$ is an extremal subgroup of $P$ with respect to $G$ and
$X$ is the set of $p$-elements of $\bN_G(Q)$, form a conjugation family.
Applying this to the subsets $\{a\}$ and $\{b\}$ of $P$, there exist extremal
subgroups $Q_1,\ldots,Q_n$ of $P$, $p$-elements $g_i\in\bN_G(Q_i)$, and an
element $k\in\bN_G(P)$ such that $g_1\cdots g_n k$ conjugates $a$ to $b$ and,
putting $w_i=g_1\cdots g_i$ and $w_0=1$, one has
$$
        a^{w_{i-1}}\in Q_i
        \qquad (i=1,\ldots,n).
$$

We prove by induction on $i$ that $w_i\in\Sub_G(a)$. This is clear for
$i=0$. Suppose that $w_{i-1}\in\Sub_G(a)$, and put $c=a^{w_{i-1}}$. Then
$c\in Q_i$, and $g_i\in\bN_G(Q_i)$ is a $p$-element. Hence
$Q_i\langle g_i\rangle$ is a $p$-group, and
$$
        \langle c,g_i\rangle\leq Q_i\langle g_i\rangle.
$$
Thus $\langle c,g_i\rangle$ is a $p$-group. Therefore $\langle c\rangle$ is
subnormal in $\langle c,g_i\rangle$, so $g_i\in S_G(c)\subseteq \Sub_G(c)$.
Moreover,
$$
        \Sub_G(c)=\Sub_G(a)^{w_{i-1}}=\Sub_G(a),
$$
because $w_{i-1}\in\Sub_G(a)$. Hence $g_i\in\Sub_G(a)$, and so
$w_i=w_{i-1}g_i\in\Sub_G(a)$.

After the last step, $a^{w_n}\in P$. Since $k\in\bN_G(P)$,
Lemma~\ref{lem:subnormalizer-picky}, applied to the element $a^{w_n}\in P$,
gives
$$
        k\in\Sub_G(a^{w_n}).
$$
As $w_n\in\Sub_G(a)$, we have
$$
        \Sub_G(a^{w_n})=\Sub_G(a)^{w_n}=\Sub_G(a),
$$
and hence $k\in\Sub_G(a)$. Therefore $w_nk\in\Sub_G(a)$ conjugates $a$ to
$b$.

Applying this with $a=x$ and $b=y^r$, we see that $x$ and $y^r$ are conjugate
in $R$. Since $r\in R$, it follows that $x$ and $y$ are conjugate in $R$.
\end{proof}

These two results led me to believe that the right subgroup in these problems
is not, in general, the Sylow normalizer, but the subnormalizer of $x$. This
was the genesis of the subnormalizer conjecture. Computations in GAP \cite{gap} provided strong evidence for it. Later, N. Rizo found simpler proofs of
these results, avoiding both Casolo's main theorem in \cite{cas} and Alperin's
Fusion Theorem.

Recall that the $p$-section of $x$ is the union of the conjugacy classes of the
elements $xy$, where $y$ runs over the $p$-regular elements of $\bC_G(x)$.
Choose representatives $y_1,\ldots,y_t$ for the $p$-regular classes of
$\bC_G(x)$, and write
$$
        S_p(x)=\{xy_1,\ldots,xy_t\}.
$$
Thus $S_p(x)$ is a set of representatives for the conjugacy classes in the
$p$-section of $x$.

Similarly, the reduced $p$-section of $x$ is the union of the conjugacy classes
of the elements $xz$, where $z$ runs over the $p$-regular elements of
$\bC_G(P)$. Choose representatives $z_1,\ldots,z_s$ so that
$xz_1,\ldots,xz_s$ represent the conjugacy classes in the reduced $p$-section
of $x$, and write
$$
        R_p(x)=\{xz_1,\ldots,xz_s\}.
$$

\begin{lemma}\label{lem:psection-representatives}
Let $x$ be a $p$-element of $G$, and let $y,z\in\bC_G(x)$ be $p$-regular
elements. If $xy$ and $xz$ are conjugate in $G$, then they are conjugate in
$\bC_G(x)$, and hence in $\Sub_G(x)$.
\end{lemma}

\begin{proof}
Suppose that $(xy)^g=xz$ for some $g\in G$. Since $x$ commutes with $y$ and
$z$, and since $y$ and $z$ are $p$-regular, the $p$-parts of $xy$ and $xz$ are
$x$ and $x$, respectively. Thus the $p$-part of $(xy)^g$ is $x^g$, and hence
$x^g=x$. Therefore $g\in\bC_G(x)$, and $xy$ and $xz$ are conjugate in
$\bC_G(x)$. Finally, $\bC_G(x)\leq \Sub_G(x)$, because
$\langle x\rangle\unlhd \langle x,c\rangle$ for every $c\in\bC_G(x)$. Hence
$xy$ and $xz$ are conjugate in $\Sub_G(x)$.
\end{proof}

We now formulate the following strengthening of the subnormalizer conjecture.

\begin{conjecture}\label{conj:reduced-p-section}
Let $G$ be a finite group, let $p$ be a prime, let $P\in \Syl_p(G)$, and let
$x\in P$. Then there exists a bijection
$$
f:\Irr^{R_p(x)}(G)\longrightarrow \Irr^{R_p(x)}(\Sub_G(x))
$$
such that
\begin{enumerate}
\item $\chi(1)_p=f(\chi)(1)_p$ for every $\chi\in \Irr^{R_p(x)}(G)$;
\item $\chi(g)=0$ if and only if $f(\chi)(g)=0$ for every $g\in R_p(x)$ and every
      $\chi\in \Irr^{R_p(x)}(G)$;
\item $\mathbb Q(\chi(g))=\mathbb Q(f(\chi)(g))$ for every $g\in R_p(x)$ and every
      $\chi\in \Irr^{R_p(x)}(G)$;
\item for every $g\in R_p(x)$, the restriction of $f$ to $\Irr^g(G)$ is a bijection
      onto $\Irr^g(\Sub_G(x))$.
\end{enumerate}
\end{conjecture}

Although the full $p$-section is too much to expect in general, stronger
properties should hold when the Sylow $p$-subgroups are abelian. For example,
in this case I expect the strong picky and subnormalizer conjectures to be
true. The next result explains why this case is closely related to McKay-type
statements. Its proof is short, but uses deep results.

\begin{proposition}\label{prop:abelian-sylow-picky-values}
Let $G$ be a finite group, let $P\in\Syl_p(G)$ be abelian, let $1\ne x\in P$ be
picky, and let $y\in\bC_G(x)$ be $p$-regular. Then
$$
        \Irr^{xy}(G)\subseteq \Irr_{p'}(G).
$$
In particular,
$$
        \Irr^x(G)=\Irr_{p'}(G).
$$
\end{proposition}

\begin{proof}
Let $\chi\in\Irr(G)$ with $\chi(xy)\ne0$, and let $B$ be the block of $G$
containing $\chi$. By \cite[Corollary~5.9]{nav98b}, the $p$-part $x$ of
$xy$ is contained in some defect group $D$ of $B$. Since $x$ is picky, the
Sylow $p$-subgroup containing $D$ is $P$, and so $D\leq P$.

By Green's theorem \cite[Corollary~4.21]{nav98b}, there exists $u\in G$ such that
$$
        P\cap P^u=D.
$$
Since $x\in D$, we have $x\in P\cap P^u$. As $x$ is picky, it follows that
$P=P^u$, and therefore $D=P$. Thus $B$ has full defect.

Since $B$ has abelian defect groups, Kessar and Malle's theorem \cite{km}
implies that every irreducible character in $B$ has height zero. Since $B$ has
full defect, this gives $\chi(1)_p=1$. Thus $\chi\in\Irr_{p'}(G)$.

This proves $\Irr^{xy}(G)\subseteq\Irr_{p'}(G)$. Conversely, for $y=1$,
since $P$ is abelian and $x$ is picky, the index $|G:\bC_G(x)|$ is prime to $p$.
Hence \cite[Corollary~4.20]{navmc} gives $\Irr_{p'}(G)\subseteq\Irr^x(G)$,
and the equality $\Irr^x(G)=\Irr_{p'}(G)$ follows.
\end{proof}

\begin{conjecture}\label{conj:abelian-sylow-p-section}
Let $G$ be a finite group with abelian Sylow $p$-subgroups. Let $P\in \Syl_p(G)$
and let $x\in P$ be picky. Then there exists a bijection
$$
f:\Irr^{S_p(x)}(G)\longrightarrow \Irr^{S_p(x)}(\Sub_G(x))
$$
such that
\begin{enumerate}
\item $\chi(1)_p=f(\chi)(1)_p$ for every $\chi\in \Irr^{S_p(x)}(G)$;
\item $\chi(g)=\pm f(\chi)(g)$ for every $g\in S_p(x)$ and every
      $\chi\in \Irr^{S_p(x)}(G)$;
\item for every $g\in S_p(x)$, the restriction of $f$ to $\Irr^g(G)$ is a bijection
      onto $\Irr^g(\Sub_G(x))$.
\end{enumerate}
\end{conjecture}

When the Sylow $p$-subgroups have order $p$, one can prove this stronger
$p$-section version completely.

\begin{theorem}\label{thm:sylow-order-p}
Let $G$ be a finite group, let $p$ be an odd prime, let $P$ be a Sylow
$p$-subgroup of $G$, and assume that $|P|=p$. Let $1\ne x\in P$. Then there
exists a bijection
$$
f:\Irr^{S_p(x)}(G)\longrightarrow \Irr^{S_p(x)}(\bN_G(P))
$$
such that
\begin{enumerate}
\item $\chi(1)_p=f(\chi)(1)_p$ for every $\chi\in \Irr^{S_p(x)}(G)$;
\item $\chi(g)=\pm f(\chi)(g)$ for every $g\in S_p(x)$ and every
      $\chi\in \Irr^{S_p(x)}(G)$;
\item $\chi(1)\equiv \pm f(\chi)(1)\pmod p$ for every
      $\chi\in \Irr^{S_p(x)}(G)$.
\end{enumerate}
\end{theorem}

\begin{proof}
Set $N=\bN_G(P)$ and $C=\bC_G(P)$. Since $|P|=p$, every element of
$P^\#$ is picky. Hence, by Lemma~\ref{lem:subnormalizer-picky},
$\Sub_G(x)=N$ for every $x\in P^\#$. Also $C=\bC_G(x)$, and the
$p$-regular elements of $C$ lie in $K=\mathbf O_{p'}(N)$, where $C=P\times K$.
Thus the representatives in $S_p(x)$ may be chosen in the form $xy$, with
$y\in K$.

The group $N$ has a normal abelian Sylow $p$-subgroup, so every irreducible
character of $N$ has $p'$-degree. Moreover, by \cite[Corollary~4.20]{navmc},
these characters do not vanish on elements of $P^\#$. Thus
$\Irr^{S_p(x)}(N)=\Irr(N)$. On the other hand, if $\chi\in\Irr(G)$ belongs to a
block of defect zero, then $\chi$ vanishes on all $p$-singular elements. Hence
the characters in $\Irr^{S_p(x)}(G)$ are contained in the blocks of $G$ with
defect group $P$.

Let $B$ be a block of $G$ with defect group $P$.  Let $b$ be a root of $B$ in $C$, and let
$1_P\times\xi\in\Irr(b)$, with $\xi\in\Irr(K)$, be the canonical character of
$B$. By the extended first main theorem, $b^G=B$ and $b^N$ is Brauer's
correspondent of $B$ in $N$; see \cite[Chapter~4]{nav98b}. Since $C\lhd N$, the
block $b^N$ is the unique block of $N$ covering $b$, by
\cite[Corollary~9.21]{nav98b}, and its irreducible characters are the characters
in $\Irr(N\mid\xi)$.

Let $T=N_\xi$, let $e=|T:C|$, and let $t=(p-1)/e$. First consider the characters
of $N$ lying over $\mu=1_P\times\xi$. Since $T/C$ is cyclic, $\mu$ extends to a
character $\widetilde\mu\in\Irr(T)$, and Gallagher's theorem and Clifford theory
give
$$
\Irr(N\mid\mu)=\{(\widetilde\mu\lambda)^N\mid \lambda\in\Irr(T/C)\}.
$$
There are $e$ such characters. By \cite[Theorem~11.13]{nav98b}, the block $B$
has $e$ nonexceptional characters, say $\chi_1,\ldots,\chi_e$. We match these
characters with the characters $(\widetilde\mu\lambda)^N$. Condition (1) is
clear, since all these characters have $p'$-degree.

Let $y\in K$. By \cite[Theorem~11.13(d)]{nav98b},
$$
\chi_i(xy)=\varepsilon_i\frac1e(1_P\times\xi)^N(y)
$$
for some sign $\varepsilon_i$. Also
$$
(1_P\times\xi)^N(y)=\sum_{\lambda\in\Irr(T/C)}(\widetilde\mu\lambda)^N(y).
$$
For fixed $y\in K\leq C$, the value $(\widetilde\mu\lambda)^N(y)$ is independent
of $\lambda$, because the characters of $T/C$ are trivial on conjugates of $y$.
Thus $(1_P\times\xi)^N(y)=e(\widetilde\mu\lambda)^N(y)$. Since $P$ is contained
in the kernel of $(\widetilde\mu\lambda)^N$, we get
$$
\chi_i(xy)=\pm(\widetilde\mu\lambda)^N(xy).
$$
This proves condition (2) for the nonexceptional characters. Finally, the
congruence in \cite[Theorem~11.13(a)]{nav98b} gives condition (3) for the
nonexceptional characters.

Now let $\lambda_1,\ldots,\lambda_t$ be representatives of the action of $T$ on
$\Irr(P)\setminus\{1_P\}$. Clifford theory gives that the characters in
$\Irr(N\mid\xi)$ not lying over $1_P\times\xi$ are
$$
(\lambda_j\times\xi)^N,\qquad j=1,\ldots,t.
$$
The block $B$ has $t$ exceptional characters, say $\theta_1,\ldots,\theta_t$, and
we match $\theta_j$ with $(\lambda_j\times\xi)^N$. Again condition (1) is clear.
For condition (2), \cite[Theorem~11.13(e)]{nav98b} gives
$$
\theta_j(xy)=\varepsilon(\lambda_j\times\xi)^N(xy)
$$
for some sign $\varepsilon$.

Finally, condition (3) for the exceptional characters follows from the last
part of \cite[Theorem~11.13(a)]{nav98b}.

Taking the union over all blocks of defect group $P$, we obtain the required
bijection between $\Irr^{S_p(x)}(G)$ and $\Irr^{S_p(x)}(N)$.
\end{proof}

The theorem is stated for odd primes because the proof uses
\cite[Theorem~11.13]{nav98b}. If $p=2$ and $|P|=2$, then $P$ is cyclic. By
Cayley's corollary to Frobenius' normal $p$-complement theorem, in the form
\cite[Corollary~1.14]{cra}, $G$ has a normal $2$-complement. This case was
considered in \cite{mr}.

On the other hand, N. Rizo considered the extension of Theorem~\ref{thm:sylow-order-p} to arbitrary
cyclic Sylow subgroups, using Dade's theory of blocks with cyclic defect group. This is used in Theorem~\ref{thm:breuer-sporadic}.

The essential counterexamples to the strong picky conjecture that we are aware
of are the simple groups with nonabelian TI Sylow $p$-subgroups and the sporadic
counterexamples found by Breuer. The strong subnormalizer conjecture fails more
often. For instance, Malle found the following example. Let $G=G_2(4)$ and
$p=2$. There is a class of elements $x$ of order $4$ such that the subnormalizer
set is already a subgroup,
$S_G(x)=\Sub_G(x)$, and $|\Sub_G(x)|=5|\bN_G(P)|$, where
$P\in\Syl_2(G)$. This class is rational, and there is a bijection satisfying
the non-strong subnormalizer conditions, but the values do not agree up to sign
on the characters of odd degree.

The evidence so far suggests that the reduced $p$-section is the right object
in general. In all the examples considered so far, the set
$\Irr^{R_p(x)}(G)$ appears to coincide with $\Irr^x(G)$.

Finally, one may also hope for a global version of the subnormalizer conjecture.
Given $x\in P$, define
$$
Q=\{\,g\in P\mid \Sub_G(g)=\Sub_G(x)\,\}.
$$
Then one may ask whether there exists a bijection
$$
f:\Irr^Q(G)\longrightarrow \Irr^Q(\Sub_G(x))
$$
satisfying the usual conditions: preservation of the $p$-parts of degrees,
preservation of fields of values, and compatibility with the subsets $\Irr^g(G)$
for all $g\in Q$.

Although it will not be used in this paper, the following description of the subnormalizer subgroup of a $p$-element, due to Malle, is very convenient.

\begin{proposition}[Malle]
Let $G$ be a finite group, let $p$ be a prime, and let $x\in G$ be a $p$-element. Then $\Sub_G(x)$ is the subgroup generated by the normalizers of the Sylow $p$-subgroups of $G$ that contain $x$.
\end{proposition}

\begin{proof}
This is \cite[Proposition~2.6]{mal}.
\end{proof}

Using this result, Malle found alternative proofs of Lemmas~\ref{lem:subnormalizer-picky} and \ref{lem:subnormalizer-fusion} (see \cite[Corollary~2.7 and Corollary~2.10]{mal}). It is interesting to note that Casolo \cite[Proposition~2.1]{cas89} had proved that for a $p$-element $x$ the subnormalizer subset of $x$ is the union of the normalizers of the radical $p$-subgroups that contain $x$, so $\Sub_G(x)$ is the subgroup generated by the normalizers of the  radical $p$-subgroups that contain $x$. On the other hand, his proof of the main theorem of \cite{cas} relies on concepts like the Quillen complex. At that time, both radical $p$-subgroups and the Quillen complex had barely appeared in the representation theory literature.

\section{Evseev's conjecture and picky elements}

Another mathematician that is relevant here is Anton Evseev. His work, in my view, has received less attention than it deserves. In \cite{evs} he proposed a very interesting refinement of McKay's conjecture.

\subsection{Evseev's refinement of McKay}

We start by recalling its statement.
Let $G$ be a finite group, let $p$ be a prime, let $P\in \Syl_p(G)$, and write $H=\bN_G(P)$. In \cite{evs}, Evseev proposed a refinement of the McKay conjecture in terms of congruences modulo characters induced from suitable local subgroups. Since this is closely related to the point of view of this paper, I recall the relevant part of his notation.

If $X$ is a finite group, let $\mathcal C(X)=\mathbb Z\Irr(X)$, and set
$$
\mathcal C_p(X)=\mathbb Z\{\chi\in \Irr(X)\mid p\mid \chi(1)\}.
$$
If $S$ is a downward closed set of subgroups of $P$, let $I(X,P,S)$ be the subgroup of $\mathcal C(X)$ spanned by the characters $\theta^X$, where $\theta\in \mathcal C(L)$ for some subgroup $L\le X$ such that $L\cap P\in \Syl_p(L)$ and $L\cap P\in S$. If $H\ge \bN_G(P)$, let
$$
S(G,P,H)=\{Q\le P\mid Q\le P^t\text{ for some }t\notin H\}.
$$
Then $(\mathrm{IRC\text{-}Syl})$ asks for a signed bijection $F:\pm\Irr_{p'}(G)\to \pm\Irr_{p'}(H)$ such that
$$
F(\chi)-\chi_H\in I(H,P,S(G,P,H))
$$
for every $\chi\in \pm\Irr_{p'}(G)$.
The weaker condition $(\mathrm{WIRC\text{-}Syl})$ asks for such a signed bijection satisfying
$$
F(\chi)-\chi_H\in \mathcal C_p(H)+I(H,P,S(G,P,H))
$$
for every $\chi\in \pm\Irr_{p'}(G)$.

Evseev was aware that $(\mathrm{IRC\text{-}Syl})$ does not hold in general, but he found only a small number of counterexamples, all of them related to twisted groups of Lie type. In particular, he pointed out that the sporadic counterexamples he found all contain one of the groups $\PSU(3,q)$ as a subgroup. His Conjecture~1.5 states that $(\mathrm{IRC\text{-}Syl})$ holds when $G$ has abelian Sylow $p$-subgroups. He also mentioned that $(\mathrm{WIRC\text{-}Syl})$ seemed to stand on less firm ground than Conjecture~1.5: since there are comparatively few cases in which $(\mathrm{IRC\text{-}Syl})$ fails, it is not easy to decide what the right weakening should be, and the correct condition might well be stronger rather than weaker than $(\mathrm{WIRC\text{-}Syl})$.

\subsection{Evseev's condition and picky elements}

The relation with the strong picky conjecture is the following.

\begin{proposition}
Assume that $(\mathrm{IRC\text{-}Syl})$ holds for $G$. Then there exists a signed bijection $f:\pm\Irr_{p'}(G)\to \pm\Irr_{p'}(\bN_G(P))$ such that for every $\chi\in \Irr_{p'}(G)$ there exists $\varepsilon_\chi\in\{1,-1\}$ with $f(\chi)(x)=\varepsilon_\chi\chi(x)$ for every picky element $x\in P$.

In particular, Evseev's Conjecture~1.5 implies this conclusion whenever $G$ has abelian Sylow $p$-subgroups.
\end{proposition}

\begin{proof}
Set $H=\bN_G(P)$, and let $f$ be a signed bijection satisfying $(\mathrm{IRC\text{-}Syl})$. It is enough to show that every element of $I(H,P,S(G,P,H))$ vanishes on the picky elements of $P$.

Let $Q\in S(G,P,H)$, and let $L\le H$ be such that $Q\in \Syl_p(L)$. Let $\delta\in \mathcal C(L)$, and let $x\in P$ be picky. If $(\delta^H)(x)\ne 0$, then by the formula for induced characters $x$ is $H$-conjugate to an element of $L$. Since $x$ is a $p$-element and $Q$ is a Sylow $p$-subgroup of $L$, it follows that $x\in Q^h$ for some $h\in H$. Hence $x\in P\cap P^{th}$ for some $t\notin H$ with $Q\le P^t$. Since $x$ is picky, this implies $P=P^{th}$, so $th\in H$, and therefore $t\in H$, a contradiction. Thus $(\delta^H)(x)=0$ for every picky element $x\in P$.

Now let $\chi\in \Irr_{p'}(G)$. Since $f(\chi)-\chi_H\in I(H,P,S(G,P,H))$, the previous paragraph gives $f(\chi)(x)=\chi(x)$ if $f(\chi)\in \Irr_{p'}(H)$ and $f(\chi)(x)=-\chi(x)$ if $f(\chi)\in -\Irr_{p'}(H)$, for every picky element $x\in P$. This proves the result.
\end{proof}

The comparison with the strong picky conjecture is quite close. In Evseev's
Tables~6.1--6.3 the failures of $(\mathrm{IRC\text{-}Syl})$ coincide with the
failures of the strong picky conjecture, except for two cases. The first is
$(J_2,3)$: here there are no picky elements, so the strong picky conjecture
holds vacuously, whereas $(\mathrm{IRC\text{-}Syl})$ fails. The second is
$(\PSU_5(2),2)$: in this case $(\mathrm{IRC\text{-}Syl})$ fails, whereas the
strong picky conjecture holds. Thus the comparison suggests that the stronger
statement envisaged by Evseev may be closer to the right one.

If $\mathcal X$ is a subset of a finite group $X$, let
$$
\mathcal C_p^{\mathcal X}(X)=
\mathbb Z\{\chi\in\Irr^{\mathcal X}(X)\mid p\mid \chi(1)\}.
$$

\begin{conjecture}
Let $\mathcal P$ be the set of picky elements of $P$. Then there exists a signed bijection $F:\pm\Irr^{\mathcal P}(G)\to \pm\Irr^{\mathcal P}(H)$ such that $F(\chi)(1)_p=\chi(1)_p$ and
$$
F(\chi)-\chi_H\in \mathcal C_p^{\mathcal P}(H)+I(H,P,S(G,P,H))
$$
for every $\chi\in \pm\Irr^{\mathcal P}(G)$.
\end{conjecture}

This may be viewed as a candidate for the kind of stronger statement that Evseev had in mind when discussing $(\mathrm{WIRC\text{-}Syl})$.

\begin{remark}
Evseev also proposed a Brou\'e-type version of his refinement. In this direction it is important to keep apart perfect isometries from weaker generalized character correspondences. Robinson gave a beautiful short argument showing that, under natural hypotheses on the Sylow normalizer, the existence of a perfect isometry between the principal block of $G$ and the principal block of $\bN_G(P)$ forces the non-exceptional characters in the principal block of $G$ to be constant and non-zero on the $p$-singular elements. He used this to rule out such a perfect isometry for the principal $2$-blocks of the Suzuki groups \cite{rob}.

Later, Eaton found that something weaker still survives. In \cite{eat}, he studied perfect generalized characters inducing the Alperin--McKay correspondence for blocks with non-abelian TI defect groups. He proved the existence of such generalized characters for non-abelian TI defect groups of order $p^3$, and also for the Suzuki groups ${}^2B_2(q)$ and the groups $\PSU_3(q)$.

This gives another indication that TI defect groups often retain substantial local-global character-theoretic structure even when the strongest possible form of a Brou\'e-type correspondence is not available. In the language of the present paper, the analogous phenomenon is that the strong picky conjecture may fail on some picky classes, while weaker value-preserving or degree-preserving correspondences may still exist.
\end{remark}

\subsection{The self-normalizing case}

We take this opportunity to present a proof of the self-normalizing case of
Evseev's condition $(\mathrm{IRC\text{-}Syl})$. We shall use the following
terminology. Recall that a subgroup $Q\leq P$ is said to be picky if
$$
        Q\leq P^g \quad\Longrightarrow\quad P^g=P
$$
for every $g\in G$. This notion can be useful in this generality. In fact, it
appears in Green's correspondence.

\begin{theorem}\label{thm:evseev-self-normalizing}
Let $G$ be a $p$-solvable group, and let $P\in\Syl_p(G)$ be self-normalizing.
Then, for every $\chi\in\Irr_{p'}(G)$, there exist a linear character
$\lambda_\chi\in\Irr(P)$ and a character (or zero) $\Delta_\chi$ such that
$$
        \chi_P=\lambda_\chi+\Delta_\chi,
$$
where $\Delta_\chi$ is a sum of characters induced from non-picky subgroups
of $P$. Moreover, the map
$$
        \chi\mapsto \lambda_\chi
$$
is a bijection from $\Irr_{p'}(G)$ onto $\Irr(P)$. In particular,
$(\mathrm{IRC\text{-}Syl})$ holds for $G$.
\end{theorem}

\begin{proof}
We prove the result by induction on $|G|$, following the proof of
\cite[Theorem~9.4]{navmc}. Let $\chi\in\Irr_{p'}(G)$. Since $\chi(1)$ is prime
to $p$, the restriction $\chi_P$ has a linear constituent. We shall denote by
$\lambda_\chi$ the linear constituent selected in the proof of
\cite[Theorem~9.4]{navmc}. We prove that
$$
        \chi_P-\lambda_\chi
$$
is a sum of characters induced from non-picky
subgroups of $P$.

Let $K=\mathbf O_{p'}(G)$. As in the proof of \cite[Theorem~9.4]{navmc}, we
have $K\leq \ker(\chi)$. Thus we may replace $G$ by $G/K$. Furthermore, if $Q\leq P$, then $Q$ is
picky as a subgroup of $G$ if and only if $QK/K$ is picky as a subgroup of
$G/K$. Thus we may assume that $K=1$.

Put $L=\mathbf O_p(G)$. Since $G$ is $p$-solvable and
$\mathbf O_{p'}(G)=1$, we have $L>1$. Let
$$
        \nu=(\lambda_\chi)_L .
$$
By \cite[Lemma~9.3]{navmc}, and since $\bN_G(P)=P$, the character $\nu$ is the
unique $P$-invariant irreducible constituent of $\chi_L$.

Let $T=G_\nu$. Then $P\leq T$. Let $\psi\in\Irr(T\mid\nu)$ be the Clifford
correspondent of $\chi$ over $\nu$. By Clifford theory, as in the proof of
\cite[Theorem~9.4]{navmc},
$$
        \chi_T=\psi+\Xi,
$$
where no irreducible constituent of $\Xi$ lies over $\nu$.

We first prove that $\Xi_P$ is a sum of characters
induced from non-picky subgroups of $P$. Let $\rho\in\Irr(P)$ be an irreducible constituent of
$\Xi_P$, and let $\eta\in\Irr(L)$ be an irreducible constituent of $\rho_L$.
Then $\eta$ is an irreducible constituent of $\chi_L$, but $\eta\ne\nu$,
because no irreducible constituent of $\Xi$ lies over $\nu$.

Let $P_\eta$ be the stabilizer of $\eta$ in $P$. By Clifford theory,
$\rho=\alpha^P$ for some $\alpha\in\Irr(P_\eta)$. We claim that $P_\eta$ is
non-picky. Since $\eta$ is an irreducible constituent of $\chi_L$, Clifford
theory gives $\eta=\nu^g$ for some $g\in G$. Hence $G_\eta=T^g$. Let
$R\in\Syl_p(T^g)$ with $P_\eta\leq R$. Since $T^g$ contains the Sylow
$p$-subgroup $P^g$ of $G$, we have $|R|=|P|$. Choose $y\in G$ such that
$R\leq P^y$. Then
$$
        P_\eta\leq P\cap P^y .
$$
If $P^y=P$, then $R=P$, and hence $P\leq T^g=G_\eta$. Thus $\eta$ would be
$P$-invariant, contradicting the uniqueness of $\nu$ as the $P$-invariant
constituent of $\chi_L$. Therefore $P^y\ne P$, and $P_\eta$ is non-picky.
It follows that $\Xi_P$ is a sum of characters induced
from non-picky subgroups of $P$.

Now \cite[Lemma~5.11]{navmc} gives that $\nu$ extends to $T$. Let
$\mu\in\Irr(T)$ be a linear extension of $\nu$ to $T$. By Gallagher's theorem,
there exists $\tau\in\Irr(T/L)$ such that
$$
        \psi=\tau\mu.
$$
Since $|T/L|<|G|$ and $P/L$ is
self-normalizing in $T/L$, the inductive hypothesis applied to $T/L$ gives
$$
        \tau_{P/L}=\tau_1+\Delta_1,
$$
where $\tau_1\in\Irr(P/L)$ is linear and $\Delta_1$ is a sum of characters induced from non-picky subgroups of $P/L$.

$$
        \psi_P=\tau_P\mu_P=\tau_1\mu_P+\mu_P\Delta_1.
$$
The character $\tau_1\mu_P$ is linear. Moreover, $\mu_P\Delta_1$ is a sum of characters
induced from non-picky subgroups of $P$.

Combining this with the previous paragraph applied to $\Xi_P$, we get
$$
        \chi_P=\psi_P+\Xi_P
        =\tau_1\mu_P+\mu_P\Delta_1+\Xi_P .
$$
Hence $\lambda_\chi=\tau_1\mu_P$ is linear and
$$
        \chi_P=\lambda_\chi+\Delta_\chi,
        \qquad
        \Delta_\chi=\mu_P\Delta_1+\Xi_P,
$$
where $\Delta_\chi$ is a sum of characters induced from non-picky subgroups
of $P$. The bijectivity of $\chi\mapsto\lambda_\chi$ is the bijectivity in
\cite[Theorem~9.4]{navmc}.

Finally, since $\bN_G(P)=P$, the subgroups of $P$ which occur in Evseev's
condition for $H=P$ are precisely the non-picky subgroups of $P$. Therefore the
decomposition above gives $(\mathrm{IRC\text{-}Syl})$ in the self-normalizing
case.
\end{proof}

\section{A comparison with the Eaton--Moret\'o conjecture}

The Eaton--Moret\'o conjecture gives another local-global prediction relating
the heights of characters in a block to the character degrees of a defect
group. I recall its statement in order to compare it with the picky conjecture.
Let $B$ be a $p$-block of a finite group $G$, and let $D$ be a defect group of
$B$. If
$\chi\in\Irr(B)$, its height is defined by
$$
        \chi(1)_p=|G:D|_p p^{h(\chi)}.
$$
Let $mh(B)$ be the smallest positive height of a character in $B$, with
$mh(B)=\infty$ if all characters in $B$ have height zero. Let $mh(D)$ be the
smallest positive integer $h$ such that $D$ has an irreducible character of
degree $p^h$, with $mh(D)=\infty$ if $D$ is abelian. The Eaton--Moret\'o
conjecture asserts that
$$
        mh(B)=mh(D).
$$
This conjecture is still not understood in general.

The picky conjecture gives a different kind of information. It is not a
blockwise statement about the first positive height. It fixes an element $x$ and
looks at the characters which do not vanish at $x$. The picky conjecture
predicts more than the equality of extremal invariants. It predicts a bijection
between two sets of characters attached to a fixed $p$-element, preserving the
$p$-parts of their degrees and the corresponding fields of values. In
particular, it immediately gives the following numerical consequence.

\begin{corollary}
Let $G$ be a finite group, let $p$ be a prime, let $P\in\Syl_p(G)$, and let
$x\in P$ be a picky element for which the picky conjecture holds. Then the
multisets
$$
        (\chi(1)_p\mid \chi\in\Irr^x(G))
        \quad\hbox{and}\quad
        (\psi(1)_p\mid \psi\in\Irr^x(\bN_G(P)))
$$
coincide.
\end{corollary}

The symmetric group case gives a concrete illustration. Let $G=S_{2^n}$,
$p=2$, and let $x$ have cycle type
$$
        (2^{n-1},2^{n-2},\ldots,2^2,2,1^2).
$$
Then $x$ is picky, and the results of \cite{mar} give
$$
        \{\chi(1)_2\mid \chi\in\Irr^x(S_{2^n})\}
        =
        \{1,2,2^2,\ldots,2^{n-2}\}
        =
        \{\lambda(1)_2\mid \lambda\in\Irr^x(P)\}.
$$
This should be compared with the set of $p$-parts of the degrees of the
irreducible characters of $S_n$. I am not aware of a simple description of
these sets in general, but they certainly have gaps, at least when $p=2$.

It seems likely that, for picky elements,
$$
        \{\chi(1)_p\mid \chi\in\Irr^x(G)\}
        =
        \{\lambda(1)_p\mid \lambda\in\Irr^x(P)\}.
$$

\section{Elements of mixed order}

\subsection{Subnormalizers of mixed-order elements}

Up to this point, the conjectures have been formulated for $p$-elements. The
guiding principle is the same: the relevant local subgroup is the subnormalizer of
the element. We notice that for elements that do not have prime power order the
subnormalizer of $x$ is usually a remarkably small subgroup, that seems to capture the character values at $x$. In fact, by \cite[Lemma~1.1]{cas89}, 
$$
\Sub_G(x)\leq \bigcap_{p\mid o(x)} \Sub_G(x_p),
$$
but it often is an even simpler group than the Sylow normalizer of any of the $p$-parts. 

\begin{conjecture}\label{con:mixed1}
Let $G$ be a finite group and let $x\in G$. Write $x=\prod_p x_p$ for the
decomposition of $x$ into its $p$-parts. Assume that, for every prime $p$ dividing
$o(x)$, the pair $(G,x_p)$ satisfies the strong subnormalizer conjecture. Then there
exists a bijection
$$
f:\Irr^x(G)\longrightarrow \Irr^x(\Sub_G(x))
$$
such that
\begin{enumerate}
\item $\chi(x)=\pm f(\chi)(x)$ for every $\chi\in \Irr^x(G)$;
\item if $x_{p'}$ centralizes a Sylow $p$-subgroup of $G$, then
$$
\chi(1)_p=f(\chi)(1)_p
$$
for every $\chi\in \Irr^x(G)$.
\end{enumerate}
\end{conjecture}

The point of this formulation is that the difficulties for mixed-order elements
should come precisely from the difficulties already present in their prime-power
parts. If one drops the hypotheses, then even the equality of cardinalities may
fail. The known counterexamples already arise from the same simple groups with
nonabelian TI Sylow subgroups that occur in the prime-power setting: for instance,
one finds failures for elements of order $6$ or $12$ in $\PSU_3(3)$, for elements
of order $10$ in $\PSU_3(4)$, and for elements of order $15$ in $McL$.

There is also a separate fusion phenomenon. The analogue of
Lemma~\ref{lem:subnormalizer-fusion} need not hold for elements of mixed order.
For instance, if
        $x=(1,2,3,4)(5,6)\in A_9$,
then $x^G\cap\Sub_G(x)$ splits into two $\Sub_G(x)$-classes. The same
phenomenon occurs for a class of elements of order $6$ in
$\mathrm{PerfectGroup}(29160,4)$. In both examples, precisely one of the
$\Sub_G(x)$-classes has the ``right values''. This is not by itself a
counterexample to the conjecture above, since the values are taken at the
chosen element $x$. Nevertheless, these examples, together with the limited evidence available so far,
make me regard Conjecture~\ref{con:mixed1} more as a question than as a conjecture. 

\subsection{Hall subgroups and sets of primes}

The same point of view also gives a natural replacement, for groups with
nilpotent Hall $\pi$-subgroups, for the false extension of McKay's conjecture
to sets of primes $\pi$. If $H$ is a nilpotent Hall $\pi$-subgroup of $G$,
one might hope for
$$
|\Irr_{\pi'}(G)|=|\Irr_{\pi'}(\bN_G(H))|.
$$
However, this is false in general.

The reason is that $\Irr_{\pi'}(G)$ is not the right set of characters to look at:
a character of $\pi'$-degree need not be nonzero on a given $\pi$-element. Thus one
is led instead to the characters that do not vanish on the relevant elements
of the Hall subgroup. If $H$ is a Hall $\pi$-subgroup of $G$, we say that
$x\in H$ is $H$-picky if, whenever $x\in H^g$, one has
$g\in \bN_G(H)$. Let
$$
\mathcal H=\{\,x\in H\mid x \text{ is } H\text{-picky}\,\}.
$$
This suggests the following conjecture.

\begin{conjecture}
Let $G$ be a finite group, let $\pi$ be a set of primes, and let $H$ be a nilpotent
Hall $\pi$-subgroup of $G$. Then there exists a bijection
$$
f:\Irr^{\mathcal H}(G)\longrightarrow \Irr^{\mathcal H}(\bN_G(H))
$$
such that
\begin{enumerate}
\item $\chi(1)_\pi=f(\chi)(1)_\pi$ for every $\chi\in \Irr^{\mathcal H}(G)$;
\item $\mathbb Q(\chi(x))=\mathbb Q(f(\chi)(x))$ for every
      $\chi\in \Irr^{\mathcal H}(G)$ and every $x\in \mathcal H$;
\item $f(\Irr^x(G))=\Irr^x(\bN_G(H))$ for every $x\in \mathcal H$.
\end{enumerate}
\end{conjecture}

This formulation appears to fix the classical counterexamples. For example,
Navarro and Tiep observed (see \cite{ns}) that the $\pi$-McKay statement fails for
$$
G=J_4,\qquad \pi=\{5,7\},
$$
where a Hall $\pi$-subgroup $H$ is cyclic of order $35$. In this example,
$$
|\Irr_{\pi'}(G)|=30\neq 25=|\Irr_{\pi'}(\bN_G(H))|.
$$
But if $x$ is a generator of $H$, then $x$ is $H$-picky and
$$
|\Irr^x(G)|=25=|\Irr^x(\bN_G(H))|.
$$
Moreover, the nonzero values at $x$ of the irreducible characters of $G$ and
of $\bN_G(H)$ occur with the same multiplicities, namely
$$
15\text{ times }\pm 1,\qquad
5\text{ times }\pm\frac{-1+\sqrt{-7}}{2},\qquad
5\text{ times }\pm\frac{-1-\sqrt{-7}}{2}.
$$
All these characters have $\pi'$-degree. Thus the failure of the classical
formulation does not come from the values at the relevant $\pi$-element. It comes
from the fact that $\Irr_{\pi'}(G)$ is too large: in this example, there are five
characters of $\pi'$-degree that vanish at $x$. The picky version removes exactly
these characters and preserves the values at $x$.

Thus, for elements of mixed order, the natural local subgroup is still
$\Sub_G(x)$, and the relevant character sets are defined by nonvanishing at the
elements under consideration, rather than by degree conditions alone.

\section{Final remarks}

I expect the conjectures in this paper to be compatible with other local-global conjectures and, in particular, with all the refinements of McKay. I will refrain from stating them. Theorem~\ref{thm:sylow-order-p} shows the compatibility of the picky conjecture with the Isaacs-Navarro refinement, in the particular case of Sylow subgroups of order $p$. 
I have not discussed block versions of these conjectures either. This was done in \cite{mr}.  

The notion of pickiness can also be considered relative to subgroups other
than Sylow subgroups. If $L\leq G$ and $x\in L$, say that $x$ is $L$-picky if
$$
        x\in L^g \quad\Longrightarrow\quad L^g=L
$$
for every $g\in G$. For $L=P\in\Syl_p(G)$ this is the notion used throughout
the paper. The previous section uses the same idea for Hall subgroups.

The interest of the sets $\Irr^x(G)$, for $x$ picky, and of the sets
$\Irr^{\mathcal P}(G)$ seems to go beyond the picky conjectures themselves.
We have already seen one instance of this. The McKay-type equality for sets of
primes is false in general, even for nilpotent Hall subgroups; see
\cite{ns}. In the picky setting the corresponding formulation uses
subnormalizers and characters not vanishing at a suitable element, and it seems
to behave better.

Here is another example. If $\mathcal X$ is a set of elements of $G$, we write
$\mathbb Q(\chi_{\mathcal X})$ for the field generated over $\mathbb Q$ by
the values $\chi(x)$, with $x\in\mathcal X$. Navarro and Tiep considered, in
their work on fields of values of characters of $p'$-degree, a condition
asserting the existence of a McKay correspondence
$$
        *:\Irr_{p'}(G)\longrightarrow \Irr_{p'}(\bN_G(P))
$$
such that
$$
        \mathbb Q(\chi_P)=\mathbb Q(\chi^*_P)
$$
for every $\chi\in\Irr_{p'}(G)$. This condition is false in general: the
counterexample found by Sambale is $\mathrm{PrimitiveGroup}(64,38)$ for $p=3$.

The global picky conjecture gives the corresponding statement for the values
on picky elements immediately.

\begin{corollary}
Let $G$ be a finite group, let $p$ be a prime, let $P\in\Syl_p(G)$, and let
$\mathcal P$ be the set of picky elements contained in $P$. Suppose that the
global picky conjecture holds for $G$, with bijection
$$
        f:\Irr^{\mathcal P}(G)\longrightarrow \Irr^{\mathcal P}(\bN_G(P)).
$$
Then
$$
        \mathbb Q(\chi_{\mathcal P})=\mathbb Q(f(\chi)_{\mathcal P})
$$
for every $\chi\in\Irr^{\mathcal P}(G)$.
\end{corollary}

\begin{proof}
For every $x\in\mathcal P$, the global picky conjecture gives
$$
        \mathbb Q(\chi(x))=\mathbb Q(f(\chi)(x)).
$$
The field generated by all the values $\chi(x)$, with $x\in\mathcal P$, is
therefore the same as the field generated by all the values $f(\chi)(x)$,
with $x\in\mathcal P$.
\end{proof}

Next, I record an extension of Corollary~\ref{cor:ratreal}. 
One consequence of the conjectures is that $\QQ_G(x)=\QQ_{\Sub_G(x)}(x)$ for any $x$ $p$-element, where $\mathbb{Q}_G(x)$ is the field generated over $\mathbb Q$ by the values $\chi(x)$ for $\chi\in\Irr(G)$. I expect this to be true for arbitrary $x$, but I can prove it for $x$ a $p$-element as a consequence of Lemma~\ref{lem:subnormalizer-fusion}.
  
  \begin{theorem}
  Let $G$ be a finite group and let $x\in G$ be a $p$-element. Then $\QQ_G(x)=\QQ_{\Sub_G(x)}(x)$.
\end{theorem}

\begin{proof}
Write $|G|=n$ and let $\sigma\in\Gal(\QQ_n/\QQ)$. Let $\chi\in\Irr(G)$. Note that $\chi(x)^{\sigma}=\chi(x)$ for every $\chi$ in $\Irr(G)$ if and only if $x$ and $x^{\sigma}$ are $G$-conjugate. Similarly, $\sigma$ fixes every character in $\Irr(\Sub_G(x))$ if and only if $x$ and $x^{\sigma}$ are $\Sub_G(x)$-conjugate. Since $x^{\sigma}$ is a power of $x$, 
$x^{\sigma}$ belongs to $\Sub_G(x)$. We conclude that $x$ and $x^{\sigma}$ are $G$-conjugate if and only if they are $\Sub_G(x)$-conjugate. Thus $\sigma$ fixes $\chi(x)$ for every $\chi\in\Irr(G)$ if and only if $\sigma$ fixes $\psi(x)$ for every $\psi\in\Irr(\Sub_G(x))$. The result now follows from the Galois correspondence.
\end{proof}

More generally, many questions about $\Irr_{p'}(G)$, or about height zero
characters, should have picky analogues. In \cite{mr} we consider picky versions
of further results and conjectures of \cite{nt}, as well as picky versions of
the Ito--Michler theorem and of Brauer's height zero theorem.

I continue with a different illustration. One of Navarro's questions asks whether
the number of Sylow $p$-subgroups of a finite group can be read from the
character table. If $x\in G$, let $\lambda_G(x)$ be the number of Sylow
$p$-subgroups of $G$ containing $x$.

\begin{theorem}
Let $G$ be a $p$-solvable group and let $P\in\Syl_p(G)$. The following are
equivalent.
\begin{enumerate}
\item The character table of $G$ determines $|\bN_G(P)|$.
\item The character table of $G$ determines the function $\lambda_G$.
\item The character table of $G$ determines $\lambda_G(x)$ for some
      nonidentity $p$-element $x\in G$.
\end{enumerate}
\end{theorem}

\begin{proof}
Let $N=\bN_G(P)$. For every $p$-element $x\in G$ we have
$$
\lambda_G(x)
=(1_N)^G(x)
=\frac{1}{|N|}
  |\{g\in G\mid x^{g^{-1}}\in N\}|.
$$
Since every $p$-element of $N$ lies in $P$, this gives
$$
        \lambda_G(x)
        =
        \frac{|\bC_G(x)|\, |x^G\cap P|}{|\bN_G(P)|}.
$$
If the character table determines $|\bN_G(P)|$, then it determines
$\lambda_G(x)$ for every $p$-element $x$, because it determines
$|\bC_G(x)|$ and, by \cite{nav98}, it determines $|x^G\cap P|$ for
$p$-solvable groups. This proves that (1) implies (2). The implication
(2) implies (3) is immediate. Conversely, if the table determines
$\lambda_G(x)$ for a nonidentity $p$-element $x$, then the same formula determines
$|\bN_G(P)|$. Thus (3) implies (1).
\end{proof}

In particular, if a picky class can be recognized from the character table, then
$\lambda_G(x)=1$ for this class, and the formula above determines
$|\bN_G(P)|$.

The examples in \cite{mmm} show that there are pairs $(G,p)$ for which $G$
has no nonidentity picky $p$-elements. This suggests the following structural
problem.

\begin{question}
Classify the pairs $(G,p)$ such that $G$ has no nonidentity picky
$p$-elements.
\end{question}

A second question concerns groups for which the subnormalizers of
$p$-elements are as large as possible.

\begin{question}
Classify the pairs $(G,p)$ such that $\Sub_G(x)=G$ for every nonidentity
$p$-element $x\in G$.
\end{question}

Malle calls this the case of almost normal Sylow $p$-subgroups; see the
discussion after \cite[Corollary~2.10]{mal}. Baumeister, Burness, Guralnick and Tong-Viet determine the finite
almost simple groups in which a Sylow $r$-subgroup is contained in a unique
maximal subgroup \cite{bbgtv} (related results have appeared in \cite{mal2}). In a major ongoing project, Guralnick and Tracey are considering the
corresponding problem for elements contained in a unique maximal subgroup
\cite{gt}. The questions above seem natural variations of these problems.

I finish by recording how the subnormalizer conjecture reduces the proof of
McKay to the proof of McKay for groups with an almost normal Sylow
$p$-subgroup.

\begin{corollary}
Assume that Conjecture~B holds for $(G,p)$. If $(G,p)$ is a minimal
counterexample to the McKay conjecture, then
$\Sub_G(x)=G$ for every nonidentity $p$-element $x\in G$.
\end{corollary}

\begin{proof}
Let $P\in\Syl_p(G)$. Since $(G,p)$ is a counterexample to McKay, $P$ is not
normal in $G$. Suppose that there exists a nonidentity $p$-element $x\in G$
such that $H=\Sub_G(x)<G$. Conjugating if necessary, we may assume that
$x\in P$. Then $\bN_G(P)\leq H$, and hence $P\in\Syl_p(H)$ and
$\bN_H(P)=\bN_G(P)$.

By Conjecture~B applied to $x$, there is a bijection
$$
        \Irr^x(G)\longrightarrow \Irr^x(H)
$$
preserving $p$-parts of character degrees. By \cite[Corollary~4.20]{navmc}, irreducible characters of $p'$-degree do
not vanish on $p$-elements. Hence
$$
        |\Irr_{p'}(G)|=|\Irr_{p'}(H)|.
$$
By the minimality of $G$, the McKay conjecture holds for $(H,p)$, so
$$
        |\Irr_{p'}(H)|=|\Irr_{p'}(\bN_H(P))|
        =|\Irr_{p'}(\bN_G(P))|.
$$
Thus McKay holds for $(G,p)$, a contradiction. Therefore
$\Sub_G(x)=G$ for every nonidentity $p$-element $x\in G$.
\end{proof}

I believe that we are still missing the underlying reason behind these
local-global phenomena. Finding it, and truly understanding why statements of this
kind should be true, seems to me an exciting problem for the coming years. 

\appendix
\section*{Appendix: A correction to a lemma of Casolo}
\addcontentsline{toc}{section}{Appendix: A correction to a lemma of Casolo}

I record a small correction to the proof of \cite[Lemma~2]{cas}, since the
formula of Casolo used in Section~4 depends on it. I thank N. Rizo,
M. Solera and G. Souza for pointing out this issue.

Let $X$ be a finite meet-semilattice, let $g$ be an automorphism of $X$, and
let $x\in X$. As in \cite{cas}, write $X^g$ for the fixed point subposet and
${}^gX$ for the downward closure of $X^g$ in $X$. Then the posets
$$
        X^g_{>x}
        \qquad\text{and}\qquad
        {}^gX_{>x}
$$
are homotopy equivalent.

In the proof of \cite[Lemma~2]{cas}, it is claimed that $X^g_{>x}$ is a
meet-semilattice. This is false in general: for instance, if
$X=\mathcal P(\{1,2\})$, ordered by inclusion, $g=1$, and $x=\varnothing$,
then $X^g_{>x}$ contains $\{1\}$ and $\{2\}$ but not their meet
$\varnothing$.

The proof is nevertheless easily repaired. Let
$$
        i:X^g_{>x}\longrightarrow {}^gX_{>x}
$$
be the inclusion. If $X^g_{>x}$ is empty, then ${}^gX_{>x}$ is empty as well.
Otherwise, by Quillen's fibre criterion it is enough to prove that, for every
$y\in{}^gX_{>x}$, the fibre
$$
        y\backslash i=\{\,z\in X^g_{>x}\mid z\geq y\,\}
$$
is contractible. This fibre is nonempty by the definition of ${}^gX$.
Moreover, if $z_1,z_2\in y\backslash i$, then $z_1\wedge z_2$ is fixed by
$g$, since
$$
        g(z_1\wedge z_2)=g(z_1)\wedge g(z_2)=z_1\wedge z_2.
$$
and satisfies
$$
        z_1\wedge z_2\geq y>x.
$$
Hence $z_1\wedge z_2\in y\backslash i$. Thus the fibre is finite, nonempty,
and closed under meets. It therefore has a least element, namely the meet of
all its elements. Hence it is contractible. Quillen's fibre criterion gives
the desired homotopy equivalence.

{\bf Acknowledgements:}
I would like to begin by thanking Attila Mar\'oti and Juan Mart\'inez Madrid. Without our joint work in \cite{mmm}, this project would not have existed. I am grateful to A. Mar\'oti for his continued encouragement and support, and to J. Mart\'inez Madrid for many useful conversations on this project. 

I thank the Department of Mathematics of the University of the Basque Country for its hospitality. The ideas that led to this project arose there, during the conference celebrating the sixtieth birthday of Pavel Shumyatsky in June 2023.

I thank Noelia Rizo for the work in \cite{mr}. In that work, we proved a few cases of some of the conjectures considered here,
obtained related results, many of which have not been mentioned in this paper,
and considered possible block versions of these conjectures.

Special thanks are due to Gunter Malle for his interest in this project and for many helpful conversations. I thank Radha Kessar, Markus Linckelmann, Geoff Robinson, Damiano Rossi, Benjamin Sambale and Mandi Schaeffer Fry for interesting conversations.

I am grateful to the GAP developers. The computations made with GAP played an essential role in the development of this project, and I would not have come up with these conjectures without them. I am especially grateful to Thomas Breuer for proving Theorem~\ref{thm:breuer-sporadic} so quickly after I asked him about it.

I acknowledge the use of ChatGPT for assistance with editing and checking some arguments during the preparation of this manuscript.

\end{document}